\documentclass[11pt]{amsart}

\usepackage[margin=1.3in]{geometry}                % See geometry.pdf to learn the layout options. There are lots.
\usepackage{amsmath}
\usepackage{hyperref}
\hypersetup{
	colorlinks=true,
	linkcolor=blue,
	citecolor=blue,
	filecolor=magenta,      
	urlcolor=blue,
	linktoc=page
}
\usepackage{amsthm}
\usepackage{amssymb}
\usepackage{enumitem}
\usepackage{placeins}
\usepackage{graphicx}
\usepackage{epstopdf}
\usepackage{xcolor}
\usepackage[square,numbers,sort&compress]{natbib}

\theoremstyle{plain}
\newtheorem{theorem}{Theorem}[section]
\newtheorem{corollary}[theorem]{Corollary}
\newtheorem{conj}[theorem]{Conjecture} 
\newtheorem{prop}[theorem]{Proposition}% reset theorem numbering for each chapter
\newtheorem{lemma}[theorem]{Lemma}

\theoremstyle{definition}

\newcommand{\Z}{\mathbb{Z}}
\newcommand{\Q}{\mathbb{Q}}
\newcommand{\N}{\mathbb{N}}
\newcommand{\R}{\mathbb{R}}
\newcommand{\p}{\mathbb{P}}
\newcommand{\E}{\mathbb{E}}

\newcommand{\indic}{\mathbf{1}}

\newcommand{\fl}[1]{{\left\lfloor #1 \right\rfloor}}

\newcommand{\sset}{\subset}
\newcommand{\lf}{\left}
\newcommand{\rg}{\right}

\newcommand{\mathand}{\;\text{and}\;}

\newcommand{\ga}{\gamma}
\newcommand{\Ga}{\Gamma}
\newcommand{\ep}{\epsilon}

\newcommand{\ka}{\kappa}
\newcommand{\de}{\delta}
\newcommand{\be}{\beta}

\newcommand{\sig}{\sigma}

\newcommand{\la}{\lambda}
\newcommand{\al}{\alpha}

\newcommand{\Rd}{\mathbb{R}^4_\uparrow}

\newcommand{\sA}{\mathcal{A}}

\newcommand{\sF}{\mathcal{F}}

\newcommand{\sL}{\mathcal{L}}

\newcommand{\sR}{\mathcal{R}}
\newcommand{\sS}{\mathcal{S}}
\newcommand{\sT}{\mathcal{T}}

\newcommand{\sY}{\mathcal{Y}}

\newcommand{\fh}{\mathfrak{h}}
\newcommand{\fg}{\mathfrak{g}}

\newcommand{\close}[1]{\mkern 1.5mu\overline{\mkern-1.5mu#1\mkern-1.5mu}\mkern 1.5mu}
\newcommand{\supp}{\text{supp}}

\usepackage{MnSymbol}

\DeclareMathOperator*{\argmax}{arg\,max}

\newcommand{\eqd}{\stackrel{d}{=}}

\newcommand{\cvgd}{\stackrel{d}{\to}}

\newcommand{\X}{\times}

\newcommand{\cvgdown}{\downarrow}

\renewcommand{\P}{\mathbb{P}}

\title{Non-uniqueness times for the maximizer of the KPZ fixed point}
\author{Duncan Dauvergne}
		
\begin{document}
	
	\maketitle
	
	\begin{abstract}
	Let $\fh_t$ be the KPZ fixed point started from any initial condition that guarantees $\fh_t$ has a maximum at every time $t$ almost surely. For any fixed $t$, almost surely $\max \fh_t$ is uniquely attained. However, there are exceptional times $t \in (0, \infty)$ when $\max \fh_t$ is achieved at multiple points. Let $\sT_k \sset (0, \infty)$ denote the set of times when $\max \fh_t$ is achieved at exactly $k$ points. We show that almost surely $\sT_2$ has Hausdorff dimension $2/3$ and is dense, $\sT_3$ has Hausdorff dimension $1/3$ and is dense, $\sT_4$ has Hausdorff dimension $0$, and there are no times when $\max \fh_t$ is achieved at $5$ or more points. This resolves two conjectures of Corwin, Hammond, Hegde, and Matetski. 
	\end{abstract}
	
	\section{Introduction}
	
	The KPZ (Kardar-Parisi-Zhang) universality class is a large collection of $1$-dimensional models of random interface growth and $2$-dimensional random metrics. The past twenty-five years have seen a period of intense and fruitful research on this class, with progress propelled by the discovery of a handful of exactly solvable models, including tasep, last passage percolation, directed polymers in a random environment, and the KPZ equation itself. See the books and expository articles \citep{ ferrari2010random,quastel2011introduction, corwin2012kardar, romik2015surprising, borodin2016lectures} and references therein for background on the KPZ universality class and related areas.
	
	Growth models in the KPZ universality class have a height function $t \mapsto h_t, t \ge 0$, which at every time $t$ returns a one-dimensional interface, often simply a continuous function from $\R \to \R$. Typically, the height function is a Markov process in $t$ started from some initial condition $h_0$. In \cite{matetski2016kpz}, Matetski, Quastel, and Remenik identified a limiting continuous-time Markov process for these height functions known as the \textbf{KPZ fixed point}. This object is a fundamental limit within the KPZ universality class, and height functions of all KPZ models should converge to it. This was shown for tasep in \cite{matetski2016kpz} and extended to other integrable models in \cite{nica2020one, dauvergne2021scaling, quastel2020convergence, virag2020heat}.
	
	The domain for the KPZ fixed point is the space of upper semicontinuous functions $h:\R \to \R \cup \{-\infty\}$
	such that $h \not \equiv -\infty$. We write $\fh_t(y; h_0)$ for the value of the KPZ fixed point at time $t \ge 0$ and location $y \in \R$ started from an initial condition $h_0$. The KPZ fixed point is well-defined for all times $t \ge 0$ as long as the initial condition grows sub-parabolically at $\infty$, see \cite{sarkar2020brownian}. Its transition probabilities can be described in terms of a Fredholm determinant, see\cite{matetski2016kpz},  or alternately through a variational problem involving a random metric introduced in \cite{DOV}. Moreover, $\fh_t(y; h_0)$ is real-valued for all $t > 0$, locally Brownian in $y$, and H\"older-$1/3^-$ continuous in $t$ for $t > 0$. 
	 
	 As with many scaling limits in probability, the KPZ fixed point exhibits remarkable fractal geometry.
	 In \cite{corwin2021exceptional}, Corwin, Hammond, Hegde, and Matetski started to investigate one of these fractal aspects related to an old conjecture of Johansson and the uniqueness of fixed point maximizers.
	 
	 \subsection{Johansson's conjecture and non-uniqueness times}
	 
	 While the construction of the KPZ fixed point is quite recent, some of its marginals have been known for much longer. For example, when the initial condition is a narrow wedge $\de_0: \R\to \R$, given by $\de_0(0) = 0$ and $\de_0(x) = -\infty$ for all other $x$, then for every fixed $t$ we have $\fh_t(y; \de_0) = t^{1/3} \sA(t^{-2/3} x) - x^2/t$, where $\sA:\R\to \R$ is an Airy (or Airy$_2$) process.
	 
	 Pr\"ahofer and Spohn \cite{prahofer2002scale} first identified the Airy process as a scaling limit of the polynuclear growth model (PNG). Shortly afterwards, Johansson \cite{johansson2003discrete} proved an analogous convergence statement for a discrete version of PNG. In that paper, Johansson also made the following conjecture.
	 
	 \begin{conj}[Conjecture, 1.5, \cite{johansson2003discrete}]\label{C:johansson}
	 	Almost surely, the parabolically shifted Airy process $x \mapsto \sA(x) - x^2$ attains its maximum at a unique point $Y \in \R$.
	 \end{conj}
 
 Johansson's interest in Conjecture \ref{C:johansson} was in understanding the point-to-line geodesic in a related KPZ model known as geometric last passage percolation. He proved that given Conjecture \ref{C:johansson}, the scaling limit of the endpoint of the point-to-line geodesic in geometric lat passage percolation is given by $Y$. Conjecture \ref{C:johansson} was proven by Corwin and Hammond \cite{CH} by showing that $\sA$ is locally absolutely continuous with respect to Brownian motion of variance $2$\footnote{Here we say a Brownian motion has variance $2$ if its quadratic variation on an interval $[a, b]$ is $2(b-a)$. Throughout the paper, whenever we refer to a Brownian motion (or a Bessel process) it will have variance $2$, as this is the convention for KPZ limits.}. Alternate proofs were found by \cite{pimentel2014location, flores2013endpoint}.

Rephrased in terms of the KPZ fixed point, Johansson's conjecture states that $\fh_1(\cdot; \de_0)$ has a unique argmax almost surely. The same turns out to be true for $\fh_1(\cdot; h_0)$ for any $h_0$ as long as $\fh_1(\cdot; h_0)$ has a maximum. This follows from an extension of the Corwin-Hammond absolute continuity result by Sarkar and Vir\'ag \cite{sarkar2020brownian}, which shows that 
$\fh_1(\cdot; h_0)$ is locally absolutely continuous with respect to Brownian continous for essentially any $h_0$.

A more nuanced situation occurs when we start to vary $t$. Indeed, it turns out that there is a fractal structure to the set of times when the KPZ fixed point attains its maximum at multiple points, first investigated in \cite{corwin2021exceptional}.
 To describe this more precisely, we first restrict to the set of initial conditions $h_0$ such that
\begin{equation}
\label{E:abstract-condition}
\P(\{\fh_t(x; h_0) : x \in \R\} \text{ has a maximum for all } t > 0) = 1.
\end{equation}
This is the set of initial conditions for which it makes sense to consider questions of non-unique maxima. 
%We call an initial condition satisfying \eqref{E:abstract-condition} \textbf{always decaying}, since $h_0$ satisfies \eqref{E:abstract-condition} if and only if for any fixed $t > 0$ almost surely we have
%\begin{equation}
%\label{E:decaying-class}
%\lim_{|x| \to \infty} \fh_t(x; h_0) \to -\infty,
%\end{equation}
%see Lemma \ref{L:unique-max}.
It is not difficult to give concrete bounds on the types of initial conditions that satisfy \eqref{E:abstract-condition}. Indeed, Lemma \ref{L:whats-in-the-set} shows that the condition
$$
\limsup_{|x| \to \infty} \frac{h_0(x)}{\log^{2/3}|x|} = - \infty
$$
implies \eqref{E:abstract-condition}, and \eqref{E:abstract-condition} implies that $h_0(x) \to -\infty$ as $|x| \to \infty$. These bounds are not optimal but give a good sense for the types of initial conditions we can work with.

Next, for a KPZ fixed point $\fh = \fh(\cdot; h_0)$ where $h_0$ satisfies \eqref{E:abstract-condition}, let
$$
\argmax \fh_t = \{y \in \R : \max_{x \in \R} \fh_t(x) = \fh_t(y)\}
$$
and for $k \in \N$ define sets of \textbf{non-uniqueness times}
\begin{equation}
\label{E:nonun-times}
\sT_k(h_0) = \{t \in (0, \infty) : |\argmax \fh_t| = k\}, \quad \sT_{\ge k}(h_0) = \{t \in (0, \infty) : |\argmax \fh_t| \ge k\}.
\end{equation} 
In \cite[Theorem 1.3]{corwin2021exceptional}, Corwin, Hammond, Hegde, and Matetski showed that for one-sided initial conditions $h_0$ satisfying a certain parabolic decay condition (Assumption 1.1 in \cite{corwin2021exceptional}), for any $b>0$ we have
\begin{equation}
\label{E:chhm-result}
\P(\dim(\sT_{\ge 2}(h_0) \cap [0, b]) = 2/3 \mid  \sT_{\ge 2}(h_0) \cap [0, b] \ne \emptyset) = 1 \quad \text{ and } \quad \P(\sT_{\ge 2}(h_0) \ne \emptyset) > 0.
\end{equation}
Here and throughout $\dim A$ refers to the Hausdorff dimension of a set $A$. They also showed that $\P(\sT_{\ge 2}(h_0) \cap [0, b] \ne \emptyset) = 1$ when $h_0$ is the narrow wedge and hence that $\P(\dim(\sT_{\ge 2}(h_0) \cap [0, b]) = 2/3) = 1$ in this special case. The heuristics in their paper suggested two natural conjectures regarding the behaviour of the $\sT_k$ that would refine \eqref{E:chhm-result}. These can be loosely summarized as follows.
\begin{conj}[Conjectures 1.5, 1.6, \cite{corwin2021exceptional}]
	\label{C:chhm}
For nice enough initial conditions $h_0$, almost surely $\sT_{\ge 2}(h_0)$ is dense in $[0, \infty)$, $\sT_{\ge 5}(h_0)$ is empty, and $\sT_{2}(h_0), \sT_3(h_0),$ and $\sT_4(h_0)$ have Hausdorff dimensions $2/3, 1/3,$ and $0$ respectively. 
\end{conj}

The main result of this paper establishes Conjecture \ref{C:chhm}.

\begin{theorem}
	\label{T:main}
	Let $\fh(\cdot; h_0)$ be a KPZ fixed point where $h_0$ satisfies \eqref{E:abstract-condition}. Then almost surely, we have the following assertions:
	\begin{enumerate}[label=\arabic*.]
		\item For any interval $I = [a, b] \sset (0, \infty)$ with $a < b$, we have $\dim(\sT_2(h_0) \cap I) = 2/3$ and $\dim(\sT_3(h_0) \cap I) = 1/3$.
		\item $\dim(\sT_4(h_0)) = 0$ and $\sT_{\ge 5}(h_0) = \emptyset$.
	\end{enumerate}
\end{theorem}
Our proof method also goes a bit further towards understanding the fractal structure of the sets $\sT_k(h_0)$ by relating them to shift and scale-invariant ergodic processes, see Lemma \ref{L:Bessel-lemma} and Section \ref{S:bessel}.

One question left open by Theorem \ref{T:main} is whether $\sT_4$ is almost surely empty. I believe this to be the case (as do \cite{corwin2021exceptional}, see the line following Conjecture 1.5 therein), but the tools in this paper are not currently refined enough to handle this border case. Note that we \emph{can} show that either $\sT_4$ is almost surely empty for every $h_0$ satisfying \eqref{E:abstract-condition}, or almost surely dense for every satisfying \eqref{E:abstract-condition}, see Proposition \ref{P:alk}.

Just as Johansson's original conjecture is connected with point-to-line geodesics, so is our Theorem \ref{T:main}. There is a way in which the maximum of the KPZ fixed point can be written as follows:
\begin{equation*}
\max_{y \in \R} \fh_t(y; h_0) = \max_{y \in \R} \max_{x \in \R} \max_{g:(x, 0) \to (y, t)} \|g\|_\sL + h_0(x).
\end{equation*}
Here the rightmost maximum is over all paths $g$ between spacetime points $(x, 0)$ and $(y, t)$, and $\|g\|_\sL$ is the length of $g$ as measured in a certain limiting random directed metric $\sL$, the directed landscape. A function-to-line geodesic (or point-to-line if $h_0$ is the narrow wedge) is any $g$ realizing the above maximum. Theorem \ref{T:main} then describes times when there are function-to-line geodesics with \emph{different} endpoints realizing this maximum. We refer the reader to Section \ref{S:prelim} and \cite{DOV, dauvergne2021scaling} for more detail on this perspective.

A broader question is to describe the set of times when there are multiple function-to-line geodesics realizing this maximum, possibly even with \emph{the same} endpoint. While this question is not explored in this paper, I believe that almost surely for all $t$, any pair of distinct function-to-line geodesics have different endpoints and so Theorem \ref{T:main} still describes Hausdorff dimensions in this broader setting.

\subsection{Proof sketch}
\label{S:proof-technique}
Unsurprisingly, the proof of Theorem \ref{T:main} can be broken into two distinct parts:
\begin{enumerate}[nosep, label=(\roman*)]
	\item Upper bounds on the Hausdorff dimension and size of the sets $\sT_{\ge k}(h_0)$ for $k = 2, 3, 4, 5$.
	\item Lower bounds on the Hausdorff dimension of the sets $\sT_{\ge k}(h_0)$ for $k = 2, 3$.
	\end{enumerate}
When thinking about (i) and (ii), I found 
that some initial conditions lent themselves to proving upper bounds and some lent themselves to proving lower bounds (and most initial conditions resisted any clean analysis at all). Because of this, a key ingredient in the paper is a transfer principle for going between initial conditions, based on the Brownian absolute continuity result of \cite{sarkar2020brownian}. 

Let us explain how the transfer principle works in a simpler setting. Let $B$ be a Brownian motion, let $A$ be a set of continuous functions, and suppose that by some $0$-$1$ law we can show that $\P(B \in A) \in \{0, 1\}$. If we can then find some process $Y$ which is absolutely continuous with respect to Brownian motion and for which $\P(Y \in A) > 0$, then $\P(X \in A) = 1$ for all processes $X$ which are absolutely continuous with respect to Brownian motion. In the setting of the KPZ fixed point there are additional technical complications, but nonetheless we can use \cite{sarkar2020brownian} to show the following for many sets $A$:
\begin{equation}
\label{E:transfer}
\text{If $\P(\fh_\cdot(\cdot; h_0) \in A) > 0$  from some $h_0$, then $\P(\fh_\cdot(\cdot; h_0) \in A) = 1$ for all $h_0$}.
\end{equation}
We prove a version of \eqref{E:transfer} tailored to our setting in Proposition \ref{P:alk}. Because of \eqref{E:transfer}, to prove our main theorems it will suffice to prove each of the statements (i) and (ii) for exactly one choice of $h_0$.

As is typical, the proof of the upper bound (i) is easier. In this case, we work with the random initial condition $h_0 = -R$, where $R$ is a two-sided Bessel-$3$ process. The advantage of this choice is that the evolution of $\fh_t(\cdot; -R), t \ge 0$ is stationary in $t$ after recentering at the maximum, see Theorem \ref{T:Bessel-IC}. Because of this, the upper bound reduces to a simple bound on a standard Bessel-$3$ process and a H\"older-$1/3^-$ continuity estimate for the KPZ fixed point. Moreover, $\fh_t(\cdot; - R)$ has scale invariance and ergodicity properties that allow us to prove the kind of $0$-$1$ laws that underlie the transfer principle \eqref{E:transfer}.

The proof of the lower bound (ii) is much harder. Here, the correct choice is to let $h_0$ consist of a sum of $k$ narrow wedges, where $k = 2, 3$. At small time scales, $\fh_t(\cdot; h_0)$ splits into $k$ independent narrow wedges, which will create competing maxima that evolve independently. The proof then boils down to estimating the transition probability  from time $t$ to time $t + s$ for the maximum of $\fh_t$ started from a single narrow wedge.

This is the sort of task for which the Fredholm determinant formula from \cite{matetski2016kpz} is ideally suited. However, there turns out to be a major difficulty here in that the estimate coming from the Fredholm determinant formula increases as a double exponential in one parameter direction.
To deal with this, we find a way of thinning out the sets $\sT_{\ge k}(h_0)$ to mitigate the effect from this weak bound. It turns out that a law of the iterated logarithm for the KPZ fixed point can be used in the thinning to just about nullify the effects of the double exponential, see Section \ref{S:lower} for details.

We note that our method is quite different than the method used in \cite{corwin2021exceptional} to prove \eqref{E:chhm-result}. In \cite{corwin2021exceptional}, the computations required for bounding Hausdorff dimension are tackled fairly directly. In this paper we take a more indirect route, using the ideas discussed above (e.g. transfer principle, finding a stationary process, set thinning) to shift difficulty away from the final computations. One advantage of this approach is that because the final computations are straightforward, we are able to strengthen \eqref{E:chhm-result} and prove Theorem \ref{T:main} without additional work.

\subsection{Some related work} This paper follows a recent line of research based around exploiting and understanding the locally Brownian structure of the Airy process $\sA$ and the KPZ fixed point $\fh_t$. For example, quantitative Brownian comparisons from \cite{dauvergne2021bulk} underlie the construction of the directed landscape \cite{DOV}.
Many other recent results about the KPZ fixed point, the directed landscape, and related objects, e.g. see \cite{hammond2020exponents, ganguly2020stability, basu2021temporal, dauvergne2021disjoint}, have been proven with the aid of quantitative comparisons between $\sA$ and Brownian motion from \cite{hammond2016brownian} and \cite{CHH20}. 
 The absolute continuity for $\fh_t$ from \cite{sarkar2020brownian} seems harder to quantify, but
nonetheless has proved useful for showing qualitative results, e.g. see \cite{dauvergne2020three}.

This paper and \cite{corwin2021exceptional} are also not the first to study fractal geometry in a KPZ limit. For example, \cite{basu2021fractal, bates2019hausdorff} use Brownian absolute  continuity results to study fractal geometry in the directed landscape. One major difference between the problem studied here and in \cite{corwin2021exceptional} and those previous works is that here, we are trying to understand what happens as the \emph{time} coordinate changes. This appears to require nuance that is very difficult to access with Brownian comparisons alone, hence the partial use of Fredholm determinants both here and in \cite{corwin2021exceptional}.

Transferring results from one initial condition to others is also a frequently used idea in KPZ, though not in the exact guise of \eqref{E:transfer}. For example, in many models it is possible to show that on small scales, height profiles from different initial conditions agree. This idea underlies the patchwork quilt description of the KPZ fixed point from \cite{hammond2019patchwork} and the papers \cite{pimentel2018local, balazs2021local, busani2020universality} which extend bounds and geodesic properties from the stationary Brownian initial condition to more general ones.  
\subsection{Brief outline of the paper}

Section \ref{S:prelim} introduces definitions and results required in the paper. Sections \ref{S:ic-abs-cont} introduces always decaying initial conditions and provides absolute continuity estimates that build on \cite{sarkar2020brownian}. Section \ref{S:bessel} focuses on the Bessel initial condition, its stationarity and ergodicity properties, and how these imply a transfer principle for Hausdorff dimension. Section \ref{S:upper} then proves the upper bounds required for Theorem \ref{T:main} and Section \ref{S:lower} proves the lower bound, up to a technical proposition that we leave for Section \ref{A:fredholm-properties}.
	
	\section{Preliminaries}
	\label{S:prelim}
	
	\subsection{The KPZ fixed point and the directed landscape}
	The KPZ fixed point can be described in two ways. The original description in \cite{matetski2016kpz} is through a Fredholm determinant formula. We will only require this formulation for the proof of one key proposition in the paper, so we will set this description aside for now and return to it when it is required in Section \ref{A:fredholm-properties}.
	
	For the most part, we will understand the KPZ fixed point by appealing to a variational formula in terms of another KPZ limit, the directed landscape.
	The \textbf{directed landscape} is a random continuous function $\sL:\Rd \to \R$, where $\Rd := \{(x, s; y, t) \in \R^4: s < t\}$. It is best thought of as a metric on the spacetime plane, where $\sL(p; q)$ is thought of as a distance from $p = (x, s)$ to $ q = (y, t)$. However, it is not a true metric: it is not necessarily positive, it is not symmetric, and it satisfies the triangle inequality backwards:
	\begin{equation}
	\label{E:triangle-ineq}
	\sL(p; q) \le \sL(p; r) + \sL(r; q).
	\end{equation}
	For any initial condition $h_0$, the KPZ fixed point $\fh(\cdot; h_0)$ can be expressed using a variational problem involving $\sL$:
	\begin{equation}
	\label{E:hopf-lax}
	\fh_t(y) = \sup_{x \in\R} h_0(x) + \sL(x, 0; y, t).
	\end{equation}
	Unless otherwise stated, we always assume $\fh_t$ and $\sL$ are related by \eqref{E:hopf-lax}.
	
	The directed landscape is built from a random continuous function $\sS:\R^2 \to \R$ known as the \textbf{Airy sheet} whose explicit description we do not require in this paper. The directed landscape is the unique function (in distribution) satisfying the following almost sure properties:
	\begin{enumerate}[label=\Roman*.]
		\item (Metric composition) For any $s < r < t$ and $x, y \in \R$ we have
		$$
		\sL(x, s; y, t) = \max_{z \in \R} \sL(x, s; z, r) + \sL(z, r; y, t).
		$$
		\item (Independent increments) For any disjoint intervals $(s_i, t_i), i = 1, \dots, k$, the random functions $\sL(\cdot, s_i; \cdot, t_i):\R^2 \to \R, i = 1, \dots, k$ are independent.
		\item (Airy sheet marginals) For any $t \in \R, s > 0$ we have
		$$
		s^{-1}\sL(s^2x, t; s^2y, t + s^3) \eqd \sS(x, y),
		$$
		where the equality in distribution is as a function of $x$ and $y$.
	\end{enumerate}
Like many scaling limits, the directed landscape has many symmetries that we will use throughout the paper, frequently without reference.
\begin{prop}[Lemma 10.2, \cite{DOV}]
		\label{P:land-props}
		We have the following equalities in distribution as random continuous functions from $\Rd \to \R$. Here $r, z \in \R$, and $q > 0$.
		\begin{enumerate}[nosep, label=\arabic*.]	
			\item KPZ scale invariance:
			$
			\sL(x, t ; y, t + s) \eqd  q \sL(q^{-2} x, q^{-3}t; q^{-2} y, q^{-3}(t + s)).
			$
			\item Time stationarity:
			$
			\sL(x, t ; y, t + s) \eqd \sL(x, t + r ; y, t + s + r).
			$
			\item Spatial stationarity:
			$
			\sL(x, t ; y, t + s) \eqd \sL(x + z, t; y + z, t + s).
			$
%			\item Flip symmetries:
%			$
%			\sL(x, t ; y, t + s) \eqd \sL(-x, s; -y, t) \eqd \sL(y, -s - t ; x, -t) 
%			$
			\item Skew stationarity:
			$$
			\sL(x, t ; y, t + s) \eqd \sL(x + zt, t; y + zt + sz, t + s) + s^{-1}[(x - y)^2 - (x - y - sz)^2].
			$$
		\end{enumerate}
%	Similarly, we have the following symmetries of the Airy sheet as random functions from $\R^2 \to \R$.
%	\begin{enumerate}[nosep, label=\arabic*.]
%		\item Flip symmetries: $\sS(x,y) \eqd \sS(-x, -y) \eqd \sS(y, x).$
%		\item Spatial Stationarity:
%		$\sS(x + z, y + r) + [(x + z) - (y + r)]^2 \eqd \sS(x, y) + (x-y)^2.
%		$
%	\end{enumerate}
\end{prop}
In particular, these symmetries imply that for any fixed $x, s, y, t$ we have $\sL(x, t; y, t + s) \eqd s^{1/3}\sL(0,0; 0, 1) - (x-y)^2/s$. The random variable $\sL(0,0; 0, 1)$ is a standard GUE Tracy-Widom random variable; this goes back to \cite{baik1999distribution, johansson2000shape}. 

In order to more easily follow arguments in the paper, the reader should also have a good sense of the rough shape of the directed landscape and some of its continuity properties. Propositions \ref{P:corollary107} and \ref{P:corollary107-airysheet} record estimates on how close the directed landscape and Airy sheet are to their expected parabolic shapes. Proposition \ref{P:modulus-bound} records a continuity estimate for $\sL$.

\begin{prop}[Corollary 10.7, \cite{DOV}]
\label{P:corollary107}
	For all $u = (x, t; y, t + s) \in \R^4_\uparrow$, we have
	$$
	\lf|\sL(x, t; y, t + s) + \frac{(x - y)^2}{s} \rg|\le C s^{1/3} \log^{4/3}\lf(\frac{2(\|u\|_2 + 2)^{3/2}}{s}\rg)\log^{2/3}(\|u\|_2 + 2).
	$$
	Here
	$C > 0$ is a random constant with $\E a^{C^{3/2}} < \infty$  for some $a > 1$.
\end{prop}

\begin{prop}[Lemma 5.3, \cite{dauvergne2020three}]
\label{P:corollary107-airysheet}
There exists a constant $c > 0$ such that for all $x, y \in \R$ we have
\begin{equation*}
|\sS(x,y)+(x-y)^2|\le C+c\log^{2/3}(2+|x|+|y|).
\end{equation*}
	Here
$C > 0$ is a random constant with $\E a^{C^{3/2}} < \infty$  for some $a > 1$.
\end{prop}
\begin{prop}[Proposition 1.6, \cite{DOV}]
	\label{P:modulus-bound}
Let $\sR(x,t;y,t+s)=\sL(x,t;y,t+s) + (x-y)^2/s$ denote the stationary version of the directed landscape. Let $K\subset \R^4_\uparrow$ be a compact set. Then
$$
|\sR(u) - \sR(v)| \le C_K \lf(\tau^{1/3}\log^{2/3}(\tau^{-1}+1) + \xi^{1/2}\log^{1/2}(\xi^{-1}+1) \rg)
$$
for all points $u = (x, s; y, t), v = (x', s'; y', t') \in K$ where $\xi := \|(x, y) - (x', y')\|_2, \tau := \|(s, t) - (s', t')\|_2$. Here
$C_K > 0$ is a random constant depending on $K$ with $\E a^{C_K^{3/2}} < \infty$  for some $a > 1$.
\end{prop}

The key thing to remember from Proposition \ref{P:modulus-bound} is that $\sL$ is H\"older-$1/3^-$ as we vary the time coordinates. This estimate is crucial for understanding how the Hausdorff dimensions in Theorem \ref{T:main} arise.

The directed landscape also has mixing properties both as we shift in time and space and as we rescale. We will only need a few of these properties moving forward.
	
\begin{prop}
		\label{P:tail-trivial}
		For $t > 0$, define $\sF_t = \sig\{\sL(x, s; y, r) : x, y \in \R, s < r \in [0, t]\}$. Then for any $A \in \sF_0^+ :=\bigcap_{t > 0} \sF_t$, we have $\P A \in \{0, 1\}$.
\end{prop}

\begin{proof}
	Almost surely, for any $(x, s; y, t)\in \Rd$ with $0 \le s$, we have
	$$
	\sL(x, s; y, t) = \lim_{r \to 0^+} \sL(x, r; y, t)
	$$
	by continuity. For $r > 0$, $\sL(x, r; y, t)$ is independent of $\sF_0^+$ by the independent increment property of $\sL$. Therefore the whole process $\sL(\cdot, s; \cdot, \cdot), 0 \le s$ is independent of $\sF_0^+$. On the other hand, $\sF_0^+$ is contained in the $\sig$-algebra generated by $\sL(\cdot, s; \cdot, \cdot), 0 \le s$, so $\sF_0^+$ is independent of itself, yielding the result.
\end{proof}

\begin{prop}
	\label{P:landscape-mixing}
	Fix $k \in \N$. For every $i = 1, \dots, k$ and $\ep > 0$, let 
	$$
	K_{i, \ep} = \{(x, s; y, t) \in \Rd: s, t \in [0, \ep], x, y \in [i - 1/4, i + 1/4]\}.
	$$ 
	We can couple $k+1$ copies of the directed landscape $\sL_0, \sL_1, \dots, \sL_k$ so that $\sL_1, \dots, \sL_k$ are independent and almost surely for all small enough $\ep > 0$ and $i = 1, \dots, k$ we have $\sL_0|_{K_{i, \ep}} = \sL_i|_{K_{i, \ep}}$.
\end{prop}

The proof of Proposition \ref{P:landscape-mixing} follows a standard, but rather tedious, approximation from a prelimiting last passage model. We leave it to Appendix \ref{A:landscape-properties}.

We will require one more result about the KPZ fixed point.
For this next theorem and throughout the paper, for random continuous functions $F, G$ with a common domain $D$ we write $F \ll_c G$ if the law of $F|_K$ is absolutely continuous with respect to the law of $G|_K$ for any compact set $K \sset D$. We will also write $X \ll Y$ for two random variables $X, Y$ if the law of $X$ is absolutely continuous with respect to the law of $Y$.

\begin{theorem}[special case of Theorem 1.2, \cite{sarkar2020brownian}]
	\label{T:BrownianAC}
	Let $\fh_\cdot$ denote the KPZ fixed point run from any initial condition $h_0$ which is bounded above. Then for any $t > 0$, we have $ \fh_t(\cdot) - \fh_t(0) \ll_c B$, where $B:\R\to \R$ is a two-sided Brownian motion.
\end{theorem}

\subsection{Absolute continuity facts}

\begin{theorem} [Theorem 2.15, \cite{dauvergne2020three}]
	\label{T:bessel} Let $B$ be a Brownian motion on $[0,1]$, and let $T$ be the unique point of maximum of $B$. Let  $\mu$ of be the law of
	$$
	B(T)-B(T+t), \qquad t\in [-T, 1 - T]\,.
	$$
	Note that $\mu$ is the law of a random function defined on random interval. 
	Let $R$ be a two-sided Bessel-$3$ process independent of $T$, and let $\nu$ be the law of $R$ on the interval $[-T,1-T]$. 
	Then $\mu$ is absolutely continuous with respect to $\nu$.
\end{theorem}

In Theorem \ref{T:bessel}, we can think of continuous functions on a random closed interval $I \sset [-1, 1]$ as continuous functions on $[-1, 1]$ by extending the functions to be constant off of $I$. Equipping these extensions with the uniform norm (and identifying functions with the same extension) makes the space of such functions a complete separable metric space.

Theorems \ref{T:BrownianAC} and \ref{T:bessel} together imply that the Bessel-$3$ process should play a crucial role in studying the KPZ fixed point near its maximum. We end this section with two more facts that will be useful for going between local absolute continuity statements and global ones.

\begin{lemma}
	\label{L:abs}
	Let $X, Y$ be two random variables, and suppose that $X_n$ is a sequence of random variables such that there exists a random $N \in \N$ with $X_N = X$ and such that $X_n \ll Y$ for all $n$. Then $X \ll Y$.
\end{lemma}

\begin{proof}
	Consider an event $A$ with $\p(X \in A) > 0$. Then there exists $n \in \N$ for which $\p(X \in A, N = n) > 0$ and hence $\P(X_n \in A) > 0$. Since $X_n \ll Y$, $\p(Y \in A) > 0$ as well.
\end{proof}

\begin{lemma}[Lemma 4.3, \cite{dauvergne2020three}]
	\label{L:abs-cvgence}
	Let $D:\R^n \to \R$ be any stochastic process, let $B$ be an $n$-dimensional Brownian motion, and suppose that $D \ll_c B$. Then
	$$
	D_\ep(t) := \ep^{-1} D(\ep^2 t)
	$$
	converges in law to $B$ as $\ep \to 0$ in the uniform-on-compact topology.
\end{lemma}

Lemma \ref{L:abs-cvgence} also applies if $B$ is replaced by a two-sided Bessel-$3$ process $R:\R \to \R$, since $R = \|B\|_2$, where $B$ is a $3$-dimensional Brownian motion.

\section{Always decaying initial conditions and absolute continuity}
	\label{S:ic-abs-cont}
We say that an initial condition $h_0:\R \to \R \cup \{-\infty\}$ is \textbf{always decaying} if for any fixed $t > 0$, a.s.\ we have
\begin{equation}
\label{E:xh-to-infty}
\lim_{|x| \to \infty} \mathfrak{h}_t(x; h_0) \to -\infty.
\end{equation}
The goal of this section is to prove a few basic properties of the KPZ fixed point run from an always decaying initial condition, culminating in two absolute continuity results relating always decaying initial conditions to the initial condition $h_0 = - R$, where $R$ is a Bessel-$3$ process. As part of this effort, we show that $h_0$ is always decaying if and only if it satisfies \eqref{E:abstract-condition}.

Our first lemma uses basic bounds to give criteria for when initial conditions are always decaying. We have not attempted to refine this lemma to get an optimal result. 

\begin{lemma}
	\label{L:whats-in-the-set}
	Let $h_0:\R \to \R \cup \{-\infty\}$ be an upper semicontinuous function.
	\begin{enumerate}[label=(\roman*)]
		\item If  $\displaystyle \lim_{|x| \to \infty} \frac{h_0(x)}{
			\log^{2/3}(|x|)} = -\infty$
		then $h_0$ is always decaying.
		\item Suppose that $\limsup_{|x| \to \infty} h_0(x) > -\infty$. Then for any $t > 0$, a.s.\ we have
		$$
		\limsup_{|x| \to \infty} \fh_t(x) = \infty.
		$$
		In particular, if $h_0$ is always decaying, then $\displaystyle \lim_{|x| \to \infty} h_0(x) = -\infty$.
	\end{enumerate}
\end{lemma}

\begin{proof}
Part (i) follows immediately from the bound in Proposition \ref{P:corollary107-airysheet} on the Airy sheet, the fact that $\sL$ has Airy sheet marginals, and the variational representation \eqref{E:hopf-lax} for $\fh_t$.

For part (ii), suppose that $\limsup_{|x| \to \infty} h_0(x) > -\infty$. Then there is an unbounded sequence $x_n, n \in \N$ and a constant $c \in \R$ such that $h_0(x_n) \ge c$ for all $n$. We have
$$
\limsup_{|x| \to \infty} \fh_t(x) \ge \limsup_{n \to \infty} \fh_t(x_n) \ge c + \limsup_{n \to \infty} \sL(x_n, 0; x_n, 1).
$$ 
Now, the process $\sL(x, 0; x, 1), x \in \R$ is stationary (Proposition \ref{P:land-props}.3) and $\sL(0, 0; 0, 1)$ is a standard Tracy-Widom and hence has unbounded support. Therefore the limsup above is $\infty$ if we can show that the joint distribution of $\sL(0,0; 0, 1), \sL(x, 0; x, 1)$ converges as $x \to \pm \infty$ to that of two independent Tracy-Widom random variables. This follows from the independence established in Proposition \ref{P:tail-trivial} and KPZ rescaling (Proposition \ref{P:land-props}.1):
\[
(\sL(0,0; 0, 1), \sL(x, 0; x, 1)) \eqd x^{1/2}(\sL(0,0; 0, x^{-3/2}), \sL(1, 0; 1, x^{-3/2})). \qedhere
\]
\end{proof}

  Next, we use the independent increment property of $\sL$ prove two results about the KPZ fixed point from an always decaying initial condition, including the equivalence of \eqref{E:xh-to-infty} and \eqref{E:abstract-condition}.
 
 \begin{lemma}
 \label{L:as-decay-wehave}
 Let $h_0$ be an always decaying initial condition. Then for any compact interval $[a, b] \sset (0, \infty)$, a.s.\ we have 
 \begin{align}
 \label{E:lim-strong}
 \lim_{|x| \to \infty} \sup_{t \in [a, b]} \fh_t(x) = -\infty.
 \end{align} 
 \end{lemma}

\begin{proof}
Suppose that the left side of \eqref{E:lim-strong} is at least some constant $c + 1$ with positive probability. For each $n \in \N$, let 
$$
T_n = \inf \{ t \in [a, b] : \sup_{x \notin [-n, n]} \fh_{t}(x) \ge c + 1/2 \}.
$$ 
Let $X_n \in [-n, n]^c$ be any (measurable) point such that $\fh_{T_n}(X_n) \ge c$. Now, since each $T_n$ is a stopping time with respect to the filtration $\sF_t = \{\sL(x, s; y, r) : r \le t\}$, by the independent increment property of $\sL$, each of the random variables $\fh_b(X_n), n \in \N$ satisfies
$$
\fh_b(X_n) \ge \fh_{T_n}(X_n) + \sL(X_n, T_n; X_n, b) \ge c + \sL(X_n, T_n; X_n, b) \eqd c + (b-T_n)^{1/3} Y,
$$
where $Y$ is an independent Tracy-Widom GUE random variable and we use the convention that $\sL(X_n, T_n; X_n, b) = 0$ if $T_n = b$. In particular,
$$
\liminf_{n \to \infty} \P(\fh_b(X_n) \ge c) > 0,
$$
and so with positive probability $\fh_b(x) \not\to -\infty$ as $|x| \to \infty$. This is a contradiction.
\end{proof}

\begin{corollary}
\label{C:equivalence}
An initial condition $h_0$ is always decaying if and only if \eqref{E:abstract-condition} holds.
\end{corollary}

\begin{proof}
Condition \eqref{E:xh-to-infty} implies condition \eqref{E:abstract-condition} by Lemma \ref{L:as-decay-wehave}. We prove the opposite implication with the contrapositive. Fix $t > 0$ and suppose that \eqref{E:xh-to-infty} does not hold almost surely at $t$. Then by Lemma \ref{L:whats-in-the-set} (ii), with positive probability $\limsup_{|x| \to \infty} \fh_{2t}(x) = \infty$, since we can view $\fh_{2t}$ as a KPZ fixed point at time $t$ started from the random initial condition $\fh_t$. This contradicts \eqref{E:abstract-condition}.
\end{proof}

We now turn our attention to proving absolute continuity results for the KPZ fixed point. The key point for these results is Theorem \ref{T:BrownianAC}, which is applicable since Lemma \ref{L:whats-in-the-set}(ii) implies that always decaying initial conditions are bounded above. For $y \in \R$ define the shift operator $T_y$ on functions $f:\R \to \R$ by
$$
T_y f(x) = f(x + y) - f(y).
$$
\begin{lemma}
\label{L:unique-max}
If $h_0$ is an always decaying initial condition, then for any fixed $t > 0$, almost surely $\fh_t = \fh_t(\cdot; h_0)$ attains its maximum at a unique point $Y_t$. Moreover, $T_{Y_t} \fh_t \ll_c -R$, where $R:\R\to \R$ is a two-sided Bessel-$3$ process.
\end{lemma}	

\begin{proof}
	By Theorem \ref{T:BrownianAC}, $T_0 \fh_t \ll_c B$, where $B$ is a Brownian motion. Therefore almost surely on every interval $[-m, m], m \in \N$, $\fh_t$ attains its maximum at a unique point $Y^m \in [-m, m]$. 
	Since $h_0$ is always decaying and $\fh_t$ is continuous, $\fh_t$ must have a global maximum. Any global argmax for $\fh_t$ must coincide with $Y^m$ for all large enough $m$. This yields the first part of the lemma, and gives that a.s.\ $Y^m = Y_t$ for all large enough $m$.
	
Next, for $m \in \N$ we define random continuous functions $f_m:\R \to \R$ as follows. First set $f_m$ equal to $T_{Y^m} \fh_t$ on $[-m - Y^m, m - Y^m]$. Next, define $f_m$ on $(-\infty, -m - Y^m] \cup [m - Y^m, \infty)$ so that the conditional law of $-f_m|_{(-\infty, -m - Y^m] \cup [m - Y^m, \infty)}$ given $f_m|_{[-m - Y^m, m - Y^m]}$ is that of a two-sided Bessel-$3$ process equal to $T_{Y^m} \fh_t(\pm m - Y^m)$ at $\pm m - Y^m$. 
	
	By Theorems \ref{T:BrownianAC} and \ref{T:bessel}, for every $m \in \N$ we have $f_m \ll - R$, where $R$ is a two-sided Bessel-$3$ process. This absolute continuity is global, not local (i.e. it is as continuous functions from $\R \to \R$).
	Moreover, for large enough $m$, on any compact set $K$ we have $f_{m}|_K = T_{Y_t} \fh_t$. Therefore by Lemma \ref{L:abs}, $T_{Y_t} \fh_t \ll_c - R$.
	\end{proof}

Next, consider an always decaying initial condition $h_0$ and let $\fh_t$ be the KPZ fixed point evolved from $h_0$ via a directed landscape $\sL$. For $t \in (0, \infty]$, define the set
\begin{align*}
\sY^t(h_0) &= \{(s, y) \in (0, t] \X \R : y \in \argmax (\fh_s(\cdot; h_0)) \}.
\end{align*}
When $t = \infty$ we simply write $\sY(h_0)$.
This next lemma concerns the distribution of the set $\sY^s(T_{Y_t}\tilde \fh_t)$. Here 
we are using an evolved KPZ fixed point $T_{Y_t}\tilde \fh_t(\cdot; h_0)$ as an initial condition, and the KPZ fixed point $\fh$ used in the definition of the set $\sY^s$ is independent of this initial condition.

\begin{lemma}
	\label{L:Bessel-lemma}
	Let $h_0$ be an always decaying initial condition, and let $R$ be a two-sided independent Bessel-$3$ process. For any $t, s \in (0, \infty)$ we have that 
	$$\sY^s(T_{Y_t}\tilde \fh_t) \ll \sY^s(-R).$$
\end{lemma}

\begin{proof}
	For ease of notation, let $F = T_{Y_t} \tilde \fh_t$. By Lemma \ref{L:unique-max}, $F \ll_c -R$. Next, for each $n$, define a continuous function $F_n:\R\to \R$ so that $F_n|_{[-n, n]} = F|_{[-n, n]}$, and the conditional law of $-F_n|_{[-n, n]^c}$ given $F_n|_{[-n, n]}$ is that of a two-sided Bessel-$3$ process whose endpoints at $\pm n$ are chosen so that $F_n$ is continuous. As in the proof of Lemma \ref{L:unique-max}, $F_n \ll -R$. 
	
	Next, we aim to show that for any $s > 0$, a.s.\ there exists $N \in \N$ for which
	\begin{align}
	\label{E:fsr}
	\qquad &\sY^s(F_N) = \sY^s(F).
	\end{align}
	If we can establish \eqref{E:fsr}, then the result will follow from Lemma \ref{L:abs}. First, observe that $F$ is a.s.\ always decaying by the definition \eqref{E:xh-to-infty} and since $\fh_{t+r}(\cdot + Y_t; h_0) - \fh_t(Y_t) \eqd \fh_r(\cdot; F)$ for any $r \ge 0$. Each $F_n$ is a.s.\ always decaying by Lemma \ref{L:whats-in-the-set}(i) and a standard bound on a Bessel process (Lemma \ref{L:bessel-1}). We require this for both sides of \eqref{E:fsr} to be a.s.\ well-defined. 
	To prove \eqref{E:fsr}, we will show that for every $\ep > 0$ we can find $m \in \N, c > 0$ such that for all large enough $n \in \N$, with probability at least $1-\ep$ all of the following statements hold:
	\begin{enumerate}[nosep, label=(\roman*)]
		\item For all $r \in (0, s]$ and $(x, y) \in \R^2$ with $|y - x| \ge c \log(2 + |x| + |y|)$ we have $\sL(x, 0; y, r) < -c < \sL(0, 0; 0, r)$. 
		\item For $r \in [0, s]$, we have
		$\sup_{x \notin [-m, m]} \fh_r(x; F) < - c$,
		\item For $r \in [0, s]$, we have $\sup_{x \notin [-m, m]} \fh_r(x; F_n) < -c$.
	\end{enumerate}
%In (ii), (iii), $\fh_r$ is the KPZ fixed point run using the landscape $\sL$.
	Let us first explain why these together imply \eqref{E:fsr}. Let $r \in (0, s]$ and assume (i)-(iii) hold for some sufficiently large $n$. First, since $F(0) = 0$, we have that $\sup_{x \in \R} \fh_r(x; F) \ge \sL(0, 0; 0, r) > -c$. Therefore
	\begin{equation}
	\label{E:hrxFF}
	\max_{x \in \R} \fh_r(x; F) > \max_{x \in A^c} \fh(x; F), \qquad \text{ where } A = \{x : \fh_r(x; F) \ge -c\}.
	\end{equation}
	By (ii), the set $A = \{x : \fh_r(x; F) \ge -c\}$ is contained in the interval $[-m, m]$. Moreover, since $F \le 0$ everywhere, condition (i) implies that at every $x \in A$ we have
	\begin{align*}
	\fh_r(x; F) &= \sup \{ F(y) + \sL(y, 0; x, r) : |y - x| \le c \log (2 + |x| + |y|) \}.
	\end{align*}
	Since $x \in [-m, m]$, as long as $n$ is sufficiently large given $m, c$ this implies
	\begin{equation}
	\label{E:hrxF}
	\fh_r(x; F) =\sup \{ F(y) + \sL(y, 0; x, r) : |y| \le n \}.
	\end{equation}
	Putting together \eqref{E:hrxFF} and \eqref{E:hrxF}, we get that 
	\begin{equation}
	\label{E:sMF}
	\sY^s(F) = 
	\sY^s(F|_{[-n, n]}),
	\end{equation} where we use the notation $F|_K$ for the function equal to $F$ on $K$ and $-\infty$ elsewhere. By analogous reasoning with (ii) in place of (iii), equation \eqref{E:sMF} also holds with $F_n$ in place of $F$. Since $F|_{[-n, n]} = F_n|_{[-n, n]}$, this implies \eqref{E:fsr}.
	
	We turn our attention to proving (i)-(iii).
	The existence of a constant $c$ satisfying (i) with probability $1-\ep/4$ follows from Proposition \ref{P:corollary107}. Given $c$, condition (ii) holds with probability at least $1-\ep/4$ for large enough $m$ by the fact that 
	$$
	\sup_{r \in [0, s], x \notin [-m, m]} \fh_{r}(x; F) \eqd  \sup_{r\in [t, t + s], x \notin [-m, m]} \fh_{r}(x + Y_t; h_0) - \fh_t(Y_t) \to - \infty
	$$
	a.s.\ as $m \to \infty$ (Lemma \ref{L:as-decay-wehave}). The proof of (iii) is slightly more involved. First, by a standard bound on Bessel processes (Lemma \ref{L:bessel-1}) for all $n \in \N$ and $x \ge n$ we have
	$$
	F_n(x) \le X_n - (x - n)^{1/4} - |F(n)|^{1/4},
	$$
	where $X_n, n \in \N$ is a tight sequence of random variables. Combining this bound with Proposition \ref{P:corollary107}, we find that there is a tight sequence of random variables $C_n > 0$ such that for all $x \ge n, y \in \R$ and $r \in [0, s]$ we have
	\begin{equation}
	\label{E:L-bd}
	\sL(x, 0; y, r) + F_n(x) \le C_n \log^2(2 + |y-n| + |x - n|) -(x-y)^2 + X_n - (x - n)^{1/4} - |F(n)|^{1/4}.
	\end{equation}
	Here we are able to center the logarithmic error at the point $(n, n)$ using the spatial stationarity of $\sL$ (Proposition \ref{P:land-props}.3). By \eqref{E:L-bd}, the tightness of $X_n, C_n$ and the fact that $F(n) \to -\infty$ almost surely we get that
	$$
	\sup_{y \in \R, x \ge n, r \in [0, s]} \sL(x, 0; y, r) + F_n(x) \to -\infty \qquad \text{in probability as $n \to\infty$.}
	$$
	A symmetric statement holds when we take the supremum over
	$y \in \R, x \le - n$. Therefore 
	$$
	\sup_{x \in \R} \fh_r(x; F_n|_{[-n, n]^c}) < - c
	$$
	for all $r \in [0, s]$ with probability tending to $1$ as $n \to \infty$. Moreover, 
	$$
	\sup_{x \notin [-m, m]} \fh_r(x; F_n|_{[-n, n]}) = \sup_{x \notin [-m, m]} \fh_r(x; F|_{[-n, n]}) \le \sup_{x \notin [-m, m]} \fh_r(x; F).
	$$
	This is less than $- c$ with probability at least $1-\ep/4$ by (ii). Combining these implies that for all large enough $n$, (iii) holds with probability at least $1 - \ep/2$, and so (i)-(iii) simultaneously hold with probability $1-\ep$, as desired.
\end{proof}
 
\section{The Bessel Initial Condition and a transfer principle}
\label{S:bessel}

Lemma \ref{L:Bessel-lemma} demonstrates the importance of the negative $2$-sided Bessel initial condition $-R$ for analyzing times where $|\argmax \fh_t| > 1$. In this section we study the KPZ fixed point from this special initial condition and use this to give a precise formulation of the transfer principle \eqref{E:transfer} in our setting.
\begin{theorem}
	\label{T:Bessel-IC}
	Let $\fh_t = \fh_t(\cdot; -R)$. For any fixed $t > 0$, almost surely $\mathfrak{h}_t$ has a unique argmax $Y_t$ and
	\begin{equation}
	\label{E:mathfrakh}
	T_{Y_t} \mathfrak{h}_t \eqd -R.
	\end{equation}
\end{theorem}

The main step in the proof of Theorem \ref{T:Bessel-IC} is the following.

\begin{lemma}
\label{L:NW-cvg}
For $t > 0$, let $Y_t = \argmax_{y \in \R} \sL(0,0; y, t)$ and let 
$$
F_{t, \ep} (x) = \ep^{-1} \lf(\sL(0,0; \ep^2 x + Y_t, t) - \sL(0,0; Y_t, t)\rg).
$$ 
Similarly, let
$$
\sL_{\ep}(x, s; y, t) = \ep^{-1} \lf(\sL(\ep^2 x + Y_1, 1 + \ep^3 s; \ep^2y + Y_1, 1 + \ep^3 t) \rg)
$$
denote the shifted and rescaled landscape. Then letting $K = \{(x, s; y, t) \in \Rd : s, t \in [0, 1]\}$,
we have
$$
(F_{1, \ep}, \sL_\ep|_K, F_{1 + \ep^3, \ep})  \cvgd (-R, \sL'|_K, -R')
$$
as $\ep \to 0$, where $R, R'$ are Bessel-$3$ processes, $\sL'$ is a directed landscape independent of $R$, and
\begin{equation}
\label{E:MC-limit}
-R' = T_{\tilde Y_1} \fh_1(\cdot; - R),
\end{equation}
where $\fh_t$ is driven by the noise $\sL'|_K$ and $\tilde Y_1$ is the a.s.\ unique argmax of $\fh_1(\cdot; -R)$. Here the underlying topology is uniform-on-compact convergence of continuous functions from $\R \X K \X \R$ to $\R$.
\end{lemma}

\begin{proof}
The convergence of the marginals $F_{1, \ep}, F_{1 + \ep^3, \ep}$ to Bessel processes $R, R'$ is immediate from Lemma \ref{L:unique-max} and the comment following Lemma \ref{L:abs-cvgence}. Also, $\sL_\ep|_K \eqd \sL'|_K$ by shift and scale invariance of $\sL$ and $F_{1, \ep}$ and $L_{\ep}|_K$ are independent for all $\ep$.

Putting all these observations together, we get that the collection $(F_{1, \ep}, \sL_\ep|_K, F_{1 + \ep^3, \ep}), \ep > 0$ is tight in $\ep$, and any subsequential distributional limit $(-R, \sL'|_K, -R')$, has the desired marginal distributions for $-R, \sL'|_K, -R'$ and has $-R, \sL'|_K$ independent. To complete the proof we just need to establish \eqref{E:MC-limit}. We start with the version of \eqref{E:MC-limit} that holds in the prelimit. Setting
\begin{equation}
\label{E:1pe}
\fh_{1, \ep}(y) := \max_{x \in \R} F_{1, \ep}(x) + \sL_\ep(x, 0; y, 1),
\end{equation}
we have
$
F_{1 + \ep^3, \ep} = T_{Y_{1, \ep}} \fh_{1, \ep},
$
where $Y_{1, \ep}$ is the (almost surely unique) argmax of $\fh_{1, \ep}$. Concretely then, it suffices to show that $(F_{1, \ep}, \sL_{1, \ep}|_K, \fh_{1, \ep}, Y_{1, \ep}) \cvgd (-R, \sL'|_K, \fh_1, \tilde Y_1)$, where $\fh_1 = \fh_1(\cdot; - R)$ and $\tilde Y_1$ are as in the statement of the lemma. Here the topology of convergence for $\fh_{1, \ep}$ is the uniform-on-compact topology.
Since $Y_{1, \ep} = \argmax \fh_{1, \ep}$ and $Y_1 = \argmax \fh_1$, we can break this convergence up into two steps:
\begin{enumerate}[nosep, label=(\roman*)]
	\item Showing $(F_{1, \ep}, \sL_\ep|_K, \fh_{1, \ep}) \cvgd (-R, \sL'|_K, \fh_1)$.
	\item Showing $(\fh_{1, \ep}, Y_{1, \ep}) \cvgd (\fh_1, \tilde Y_1)$.
\end{enumerate}
For (i), for $y \in \R$ let $A_{\ep, y} = \argmax_{x \in \R} F_{1, \ep}(x) + \sL_\ep(x, 0; y, 1)$. By the joint uniform-on-compact convergence of $F_{1, \ep}, \sL_{\ep}|_K$, to prove (i) it suffices to show that the collections of random variables 
$$
\inf \{x \in \bigcup_{y \in [a, b]} A_{\ep, y} \}, \qquad \sup \{x \in \bigcup_{y \in [a, b]} A_{\ep, y} \}, \qquad \ep > 0
$$
are tight for every compact interval $[a, b]$. This follows from combining the following three observations.
\begin{enumerate}[nosep, label=(\alph*)]
	\item $F_{1, \ep} \le 0$ for all $\ep$.
	\item $|\sL_\ep (x, 0; y, 1) +(x-y)^2| \le C_\ep + c\log^{2/3}(2 + |x| + |y|)$ for a collection of identically distributed random variables $C_\ep,\ep > 0$ and some $c > 0$. This uses scale invariance of $\sL$ and Proposition \ref{P:corollary107-airysheet}.
	\item For any $a < b$, the collection of random variables $\inf_{y \in [a, b]} F_{1, \ep}(y), \ep > 0$ is tight in $(-\infty, 0]$ since $F_{1, \ep} \cvgd -R$, uniformly on compact sets.
\end{enumerate}
We move on to (ii). By the convergence in (i), $\fh_{1, \ep} \cvgd \fh_1$ where $\fh_1$ is a copy of the KPZ fixed point started from $-R$. The function $-R$ is almost surely an always decaying initial condition by Lemma \ref{L:whats-in-the-set} and standard Bessel properties (Lemma \ref{L:bessel-1}), and hence the argmax $\tilde Y_1$ of $\fh_1$ is a.s.\ unique. Point (ii) will then follow if we can show that $Y_{1, \ep}, \ep>0$ is tight. Define
$$
X_{s, \ep} = \sup \{x \in \R : F_{1, \ep}(x) \ge  - |x|^{1/5} \}, \qquad X_{i, \ep} = \inf \{x \in \R : F_{1, \ep}(x) \ge  - |x|^{1/5} \}.
$$
By the tail bounds in point (b) above on $\sL_\ep$, to prove that $Y_{1, \ep}, \ep \in (0, 1)$ is tight it suffices to show that $X_{s, \ep}, X_{i, \ep}, \ep \in (0, 1)$ are both tight. The arguments are symmetric, so we focus on $X_{s, \ep}$. First, $F_{1, 1}$ is simply $\sL(0, 0; \cdot, 1)$, shifted and recentered at its maximum. Therefore by Proposition \ref{P:corollary107-airysheet} we have
$$
F_{1, 1}(x) < - x^{3/2}
$$
for all $x > C$ for some random $C > 0$. Now fix $\de > 0$ and choose $c > 0$ so that $\P(C > c) \le \de/2$. Using that  
$
F_{1, \ep}(x) = \ep^{-1} F_{1, 1} (\ep^2 x),
$  
on the event $\{C > c\}$, for $x > c \ep^{-2}$ we have
$$
F_{1, \ep}(x) \le -\ep^{-1} \ep^{3} x^{3/2} \le -x^{1/2} < -x^{1/5}.
$$
Hence $X_{s, \ep} < c \ep^{-2}$ on the event $\{C > c\}$. On the other hand, the collection of random variables
$$
X^*_{s, \ep} := \sup \{x \in [-c\ep^2, c \ep^2] : F_{1, \ep}(x) \ge  - |x|^{1/5} \}
$$
is tight for $\ep \in (0, 1)$ by the local absolute continuity of $F_{1, 1}$ with respect to $-R$ (Lemma \ref{L:unique-max}), and the bound in Lemma \ref{L:bessel-1} on the growth of Bessel processes. Therefore for some constant $c' > 0$, $\p(X_{s, \ep}^* > c') \le \de/2$ for all $\ep \in (0, 1)$ and so
$$
\p(X_{s, \ep} > c') \le \p(C > c) + \p(X_{s, \ep}^* > c') \le \de,
$$ 
for all $\ep > 0$, yielding the desired tightness.
\end{proof}

\begin{proof}[Proof of Theorem \ref{T:Bessel-IC}]
The $t=1$ case follows from \eqref{E:MC-limit} in Lemma \ref{L:NW-cvg} and the general case follows from KPZ scale invariance of $(\sL, - R)$.
\end{proof}

Theorem \ref{T:Bessel-IC}, along with scale invariance of $\sL$ and $R$ give corresponding properties of the set $\sY(-R)$.
\begin{prop}
\label{P:special-IC}
Let $s > 0$, and let $Y_s = \argmax \fh_s(\cdot; -R)$. Then
$$
\sY(-R) \eqd \sY_{s, Y_s}(-R) :=\{(t-s, x - Y_s): (t, x) \in \sY(-R), t > s\}. 
$$
Also, for $q > 0$ define $\sL_q(x, s; y, t) = q^{-1} \sL_q(q^2 x, q^3 s; q^2 y, q^3 t), R_q(t) = q^{-1} R(q^2 t)$, and 
$$
\sY_q(-R) := \{(q^3 s, q^2 x) : (s, x) \in \sY_q(-R)\}.
$$
Then $(\sL, R, \sY(-R)) \eqd (\sL_q, R_q, \sY_q(-R))$.
\end{prop}

\begin{proof}
The first equality in distribution is immediate from Theorem \ref{T:Bessel-IC} and the independent increment property for $\sL$, which guarantees that values of $\sL$ at times in $[s, \infty)$ are independent of $\fh_s(\cdot; -R)$. The equality in distribution $(\sL, R) \eqd (\sL_q, R_q)$ follows from Brownian scale invariance of $R$ and KPZ scale invariance of $\sL$, and the stronger claim that $(\sL, R, \sY(-R)) \eqd (\sL_q, R_q, \sY_q(-R))$ follows since there is a measurable function $f$ such that $\sY_q(-R) = f(\sL_q, R_q)$ for any $q > 0$.
\end{proof}

We also have mixing and ergodicity properties under the shifts introduced in Proposition \ref{P:special-IC}. We will only need one such result here.
\begin{lemma}
\label{L:0-1law}
The sequence of processes
$
(\sL_q, R_q, \sY_q(-R)), q \in \{2^k: k \in \Z\}
$
is ergodic.
\end{lemma}

\begin{proof}
This follows from Proposition \ref{P:special-IC}, Blumenthal's $0$-$1$ law for the Bessel-$3$ process, and Proposition \ref{P:tail-trivial}.
\end{proof}

Next we turn to the main object of the paper: non-uniqueness times for the KPZ fixed point. Recalling the notation $\sT_{\ge k}(h_0)$ from the introduction, we have the following result.

\begin{prop}
	\label{P:alk}
For every $k \in \N$, there exists a constant $\al_k \in [0, 1]$ and values $\be_k \in \{0, \infty\}$ such that for any always decaying initial condition $h_0$, a.s.\ for every interval $I = [a, b] \sset [0, \infty], a < b$ we have
\begin{align}
\label{E:dim-sTk}
\dim(\sT_{\ge k}(h_0) \cap I) &= \al_k, \qquad \mathand \qquad 
 \#(\sT_{\ge k}(h_0) \cap I) = \be_k.
\end{align}
Here $\# A$ is the cardinality of $A$ if $A$ is finite, and equal to $\infty$ if $A$ is infinite.
\end{prop}

\begin{proof}
We begin by proving the proposition for the random initial condition $h_0 = -R$ and restricting to the case when $b \ne \infty$. We start with the first equality in \eqref{E:dim-sTk}.
It suffices to prove that the statement holds a.s.\ for every rational interval $[a, b]$, since Hausdorff dimension is monotone with respect to set inclusion. For any $k \in \N$ and $b > 0$, we have
\begin{equation}
\label{E:sTkR}
\dim(\sT_{\ge k}(-R) \cap [0, b]) = \sup_{q = 0,1, 2, \dots} \dim(\sT_{\ge k}(-R) \cap [0, 2^{-3q} b]).
\end{equation}
By Lemma \ref{L:0-1law}, and the fact that Hausdorff dimension is unchanged under linear dilations of a set, the sequence $\dim(\sT_{\ge k}(-R) \cap [0, 2^{-3q} b]), q \in \N$ is ergodic. Therefore
a.s.,
\begin{equation}
\label{E:sup-eqn}
\sup_{q \in \N} \dim(\sT_{\ge k}(-R) \cap [0, 2^{-3q} b]) = \al_k,
\end{equation}
where $\al_k$ is the supremum of the support of the random variable $\dim(\sT_{\ge k}(-R) \cap [0, b])$. Moreover, \eqref{E:sup-eqn} still holds if we replace $0, 1, \dots$ by any sequence $\{n_0, \dots, \}$ for any $n_0 \in \Z$, so $\al_k$ is a.s.\ the Hausdorff dimension of $\dim(\sT_{\ge k}(-R) \cap [0,b'])$ for any $b' > 0$ as well. The result transfers to general intervals $[a, b]$ by time stationarity of $\sT_{\ge k}(-R)$, Proposition \ref{P:special-IC}.

We modify the above proof to establish the second equality in \eqref{E:dim-sTk} when $h_0 = - R$. First, the function $\#$ is monotone with respect to set inclusion, and unchanged under linear dilations of a set, so the proof above goes through verbatim to show that there exists $\be_k \in \{0, 1, \dots, \infty\}$ such that the second equality in \eqref{E:dim-sTk} holds with $h_0 = - R$
for any interval $I$. We can rule out the possibility that $\be_k \notin \{0, \infty\}$ since by the inequality 
$$
\#(\sT_{\ge k}(h_0) \cap [0, 1]) + \#(\sT_{\ge k}(h_0) \cap [2, 3]) \le \# (\sT_{\ge k}(h_0) \cap [0, 3]),
$$
we have $2 \be_k \le \be_k$.

Now consider a general initial condition $h_0$. By the result for $h_0=-R$ and the absolute continuity in Lemma \ref{L:Bessel-lemma}, a.s.\ \eqref{E:dim-sTk} holds for any rational interval $[a, b]$ with $a > 0$. The result for general $a, b$ (allowing for $a = 0$ or $b = \infty$) follows from countable stability and monotonicity of the functions $\dim$ and $\#$.
\end{proof}

Proposition \ref{P:alk} allows us to transfer information about Hausdorff dimension and set size from one initial condition to all others. In particular, Proposition \ref{P:alk} shows that if we can find always decaying initial conditions $h, h'$ such that 
$$
\p(\dim(\sT_{\ge k}(h)) \ge \al) > 0 \qquad \text{ and } \qquad \p(\dim(\sT_{\ge k}(h')) \le \al) > 0,
$$
then $\al_k = \al$.
 This can be thought of as version of the transfer principle \eqref{E:transfer} tailored to studying Hausdorff dimension.

\section{The upper bound}
\label{S:upper}

\begin{theorem}
	\label{T:upper-bd}
	With notation as in Proposition \ref{P:alk}, we have $\al_2 \le 2/3, \al_3 \le 1/3, \al_4 = 0,$ and $\be_5 = 0$.
\end{theorem}
To get the upper bound, we work with the stationary initial condition $h_0 = - R$. We will need two results:
\begin{itemize}[nosep]
	\item A H\"older continuity estimate on the KPZ fixed point started from $-R$, and
	\item A bound on the probability that a Bessel process has multiple near-minima.
\end{itemize}
We start with the H\"older continuity estimate on $\fh_t$. This is similar to \cite[Lemma 3.3]{corwin2021exceptional} and is a fairly straightforward consequence of the regularity in Propositions \ref{P:corollary107} and \ref{P:modulus-bound}.
\begin{lemma}
\label{L:Holder-1/3}
Let $h_0$ be an always decaying initial condition and let $b > 0$. Then there exists a random $B \in (0, \infty)$ such that
\begin{equation}
\label{E:hsh0}
\fh_s(x, h_0) = \fh_s(x, h_0|_{[-B, B]}) \qquad \text{ for all } s \in [0, 1], x \in [-b, b].
\end{equation}
Here the initial condition $h_0|_K$ equals $h_0$ on $K$ and $-\infty$ off $K$.
In particular, there exists a random  constant $C \in (0, \infty)$ such that
\begin{align*}
\lf|\mathfrak{h}_s(x; h_0) - \mathfrak{h}_r(x; -h_0) \rg| &\le \max_{|y| \le B} |\sL(y, 0; x, s) - \sL(y, 0; x, r)| \\
&\le C \log^{2/3}(2|r - s|^{-1})(r - s)^{1/3},
\end{align*}
for all $x \in [-b, b]$ and $1/2 \le s < r \le 1$. 
\end{lemma}

\begin{proof}
Equation \eqref{E:hsh0} is exactly what we showed in \eqref{E:hrxF}; we give proof details there. Note that in establishing \eqref{E:hrxF} we only used that the initial condition was bounded above, which for always decaying initial conditions follows from Lemma \ref{L:whats-in-the-set}(ii). The second display then follows from \eqref{E:hsh0} and Proposition \ref{P:modulus-bound}. Here we have restricted to the region where $r, s \in [1/2, 1]$ so that all the points $(y, 0; x, s), (y, 0; x, r)$ lie in a common compact set.
\end{proof}

We now address the probability that a Bessel process has multiple near-minima.
\begin{lemma}
\label{L:Bessel-bd}
Fix $k \in \N$ and a $k$-tuple of disjoint closed intervals $I = (I_1, \dots, I_k)$, where $I_j = [a_j, b_j]$, ordered so that $b_{i-1} < a_i$ for all $i \ge 2$. Set
$
m_I = \min \{a_i - b_{i-1} : i \in \{2, \dots, k\}\}.
$ 
Let $R$ be a two-sided Bessel-$3$ process and define the event
$$
A(R, I, \ep) = \{\forall  i \in \{1, \dots, k\}, \min_{x \in I_i} R(x) \le \ep \}.
$$
Then for all $\ep > 0$ we have
$
\p A(R, I, \ep) \le (\ep /\sqrt{m_I})^{k-1}.
$
\end{lemma}

\begin{proof}
Let $j \in \{1, \dots, k\}$ be chosen so that $d(I_j, 0)$ is minimized, and let $I^- = (I_1, \dots, I_{j-1})$ and $I^+ = (I_{j+1}, \dots, I_k)$. By construction, all intervals in $I^-$ are contained in $(-\infty, -m_I/2]$ and all intervals in $I^+$ are contained in $[m_I/2,\infty)$.
Therefore since $A(R, I, \ep) \sset A(R, I^-, \ep) \cap A(R, I^+, \ep)$ and $R(\cdot)|_{[0, \infty)}$ and $R(- \; \cdot)|_{[0, \infty)}$ are independent Bessel processes, we have that 
$$
\P A(R, I, \ep) \le \P A(R, I^-, \ep) \P A(R, I^+, \ep),
$$
and the lemma will follow if we can show that $\P A(R, I^-, \ep) \le (\ep /\sqrt{m_I})^{j-1}$ and $\P A(R, I^+, \ep) \le (\ep /\sqrt{m_I})^{k-j}$. The proofs of these two claims are identical, so we just focus on the latter. 

Letting $L_i = \min_{x \in I_i} R(x)$, it is enough to show that for every $i = j+1, \dots, k$ we have
\begin{equation}
\label{E:Micep}
\P( L_i < \ep \; | \; R|_{[0, b_{i-1} \vee 0]}) \le \ep /\sqrt{m_I}.
\end{equation}
First, by the Markov property for $R$, $L_i$ only depends on $R|_{[0, b_{i-1} \vee 0]}$ through $R(a_i)$. Moreover, $L_i \ge L_i' := \min_{x \in [a_i, \infty)} R(x)$ and the distribution of $L_i'$ given $R(a_i)$ is $U R(a_i)$, where $U$ is an independent uniform random variable on $[0, 1]$, see \cite[Chapter VI, Corollary 3.4]{revuz2013continuous}. Next, the conditional distribution of $R(a_i)$ given $R|_{[0, b_{i-1} \vee 0]}$ stochastically dominates the (unconditional) distribution of $R(a_i - b_{i-1} \vee 0)$. In particular, using the representation of $R$ as the magnitude of $3$-dimensional Brownian motion of variance $2$, we have
$$
R(a_i) \succeq \sqrt{2a_i - 2b_{i-1}  \vee 0}\sqrt{N_1^2 + N_2^2 + N_3^2} \ge \sqrt{m_I} \sqrt{N_1^2 + N_2^2 + N_3^2}. 
$$
where the notation $\succeq$ denotes stochastic domination given $R|_{[0, b_{i-1} \vee 0]}$, and the $N_i$ are independent standard normals. Therefore 
$$
\P( L_i < \ep \; | \; R|_{[0, b_{i-1} \vee 0]}) \le \P(U\sqrt{m_I} \sqrt{N_1^2 + N_2^2 + N_3^2}  \le \ep),
$$
and a computation then yields \eqref{E:Micep}.
\end{proof}

\begin{corollary}
	\label{C:shifted-bessel-bd}
	Let all notation be as in Lemma \ref{L:Bessel-bd}, and for a $k$-tuple $I = (I_1, \dots, I_k)$, write $I + b := (I_1 + b, \dots, I_k + b)$. 
	Then for all $\ep, b > 0$ we have
	$$
	\p \lf(\bigcup_{x \in [-b, b]} A(R, I + x, \ep) \rg) \le (8b/m_I + 2)(\ep\sqrt{2/m_I})^{k-1}.
	$$
\end{corollary}

\begin{proof}
First, define $\tilde I$ so that $\tilde I_j = [a_j - m_I/4, b_j + m_I/4]$. Then 
$$
\bigcup_{x \in [-b, b]} A(R, I + x, \ep) \sset \bigcup_{x \in [-b, b] \cap m_I \Z/4 } A(R, \tilde I + x, \ep).
$$
The probability of the latter union can be bounded with a union bound and Lemma \ref{L:Bessel-bd}, using that $m_{\tilde I} = m_I/2$.
\end{proof}

\begin{proof}[Proof of Theorem \ref{T:upper-bd}]
We first prove the bounds on $\al_k$ for $k = 2, 3, 4$. By Proposition \ref{P:alk} it is enough to show that a.s.\, for each of these values of $k$ we have
\begin{equation}
\label{E:dimdim}
\dim(\sT_{\ge k}(-R) \cap [2/3, 1]) \le (4-k)/3.
\end{equation}
To prove this, it is enough to show that for every $k$-tuple of ordered disjoint closed rational intervals $I = (I_1, \dots, I_k)$, a.s.\ we have
\begin{align}
\label{E:dimsTk}
\dim(\sT_k(-R, I)& \cap [2/3, 1]) \le (4-k)/3, \quad \text{ where } \quad \\
\nonumber
\sT_k(-R, I) &:= \{t \in (0, \infty) : \forall i = 1, \dots, k \text{ we have } \sY(-R) \cap (\{t\} \X I_i) \ne \emptyset \}. 
\end{align}
Indeed, \eqref{E:dimdim} follows from \eqref{E:dimsTk} by countable stability of Hausdorff dimension and the fact that $\sT_{\ge k}(-R) = \bigcup_I \sT_k(-R, I)$, where the union is over all ordered $k$-tuples of disjoint closed rational intervals. 

Fix $b > 0, \de \in (0, 1/3)$, and $I = (I_1, \dots, I_k)$ such that $I_j \sset [-b, b]$ for all $j$. Let $\fh = \fh(\cdot; -R)$. By Lemma \ref{L:Holder-1/3}, almost surely there exists a constant $C > 0$ such that
$$
|\fh_s(x) - \fh_r(x)| \le C(r-s)^{1/3-\de}
$$
for all $x \in [-b, b], s < r \in [1/2, 1]$. In particular, for all small enough $\ep > 0$ we have that
\begin{align}
\label{E:sTkRI}
\sT_k(-R, I) \cap [2/3, 1] &\sset \bigcup_{i \in \ep^{3 + \de} \Z \cap [1/2, 1]}  [i, i + \ep^{3 + \de}]\indic(G(\fh_i, I, \ep)), \qquad \text{ where} \\
\nonumber
 G(\fh_i, I, \ep) &= \Big\{ \max_{x \in I_j} \fh_i(x) \ge \max_{x \in \R} \fh_i(x) - \ep \text{ for all } i =1, \dots, k \Big\}.
\end{align}
Now, let
$$
B = \sup \{x \ge 0 : \fh_s(z) = \max_{y \in \R} \fh_s(y) \text{ for some } s \in [1/2, 1], z \in \{-x, x\} \}.
$$
By Lemma \ref{L:as-decay-wehave}, $B \in (0, \infty)$ almost surely. On the event where $B \le a$, we have
\begin{equation}
\label{E:Gh-set}
G(\fh_i, I, \ep) = A(T_{Y_i} \fh_i, I - Y_i, \ep) \sset \bigcup_{x \in [-a, a]} A(T_{Y_i} \fh_i, I - x, \ep). 
\end{equation}
where the notation is as in Theorem \ref{T:Bessel-IC} and Lemma \ref{L:Bessel-bd}. Now, by Theorem \ref{T:Bessel-IC} and Corollary \ref{C:shifted-bessel-bd}, the probability of the event on the right side of \eqref{E:Gh-set} is bounded above by $c\ep^{k-1}$ for some $c > 0$ depending on $I, a, k$. Therefore for any $\ga > 0$ by a union bound we have that
\begin{align}
\label{E:bound}
\E \Big( \sum_{i \in \ep^{3 + \de} \Z \cap [1/2, 1]}  (\ep^{3 + 10 \de})^\ga \indic(G(\fh_i, I, \ep)) \Big \mid B \le a \Big) \le c \ep^{3\ga + \de \ga + k - 1 - 3 - \de}.
\end{align} 
For $k= 2, 3, 4$, the right side above tends to $0$ with $\ep$ as long as $\ga > (4-k)/3 + \de/3$. Since $B < \infty$ almost surely, by Markov's inequality this gives that the Hausdorff dimension of $\sT_k(-R, I) \cap [2/3, 1]$ is almost surely bounded above by $(4-k)/3 + \de/3$. Taking $\de \to 0$ completes the proof.

To show that $\be_k = 0$, by Proposition \ref{P:alk} it is enough to show that almost surely, $\sT_5(-R) \cap [2/3, 1] = \emptyset$. Again, it is enough to show that $\sT_5(-R, I) \cap [2/3, 1] = \emptyset$ a.s.\ for every collection $I = (I_1, \dots, I_5)$ of ordered disjoint rational intervals. The identities \eqref{E:sTkRI} and \eqref{E:bound} still hold here. Set $\ga = 0$ in \eqref{E:bound}. Then the left hand side of \eqref{E:bound} is bounded below by 
$$
\p(\sT_5(-R, I) \cap [2/3, 1] \ne \emptyset \; | \; B \le a)
$$
and the right hand side of \eqref{E:bound} tends to $0$ with $\ep$ for $\de < 1/3$.  Therefore by Markov's inequality, $\sT_5(-R, I) \cap [1/2, 1]= \emptyset$ a.s.
\end{proof}

\section{The lower bound}
\label{S:lower}

\begin{theorem}
	\label{T:lower-bd}
	We have $\al_2 \ge 2/3$ and $\al_3 \ge 1/3$.
\end{theorem}

Unsurprisingly, the lower bound is more difficult. To achieve it, we work with superimposed collections of narrow wedges. 
Let $\de_x$ be the initial condition with $\de_x(x) = 0$ and $\de_x(y) = -\infty$ for all $y \ne x$, and define
\begin{equation}
h_2 = \de_1 \vee \de_2, \qquad h_3 = \de_1 \vee \de_2 \vee \de_3.
\end{equation}
In the remainder of this section we write $k$ for either $2$ or $3$. For small values of $t$, the process $\mathfrak{h}_t(\cdot; h_k)$ should have $k$ competing maxima, given by the maxima of $k$ independent KPZ fixed points started from $k$ independent narrow wedges. The next lemma makes this intuition precise.

\begin{lemma}
\label{L:h2-behaviour}
We can couple KPZ fixed points $\mathfrak{h}_t(\cdot; h_k)$ and $\fh^1_t(\cdot; \de_1), \dots, \fh^k_t(\cdot; \de_k)$ so that $\fh^1_t(\cdot; \de_1), \dots, \fh^k_t(\cdot; \de_k)$ are all independent and there exists a random $T > 0$ such that for $t < T$ we have
\begin{equation}
\label{E:k-competing}
\max_{x \in \R} \mathfrak{h}_t(x; h_k) = \max_{i=1, \dots, k} \max_{x \in \R} \fh_t^i(x, \de_i).
\end{equation}
\end{lemma}

\begin{proof}
First, we can couple $\fh_t, \fh^1_t, \dots, \fh^k_t$ so that $\fh^1_t, \dots, \fh^k_t$ are independent and there exists $T' > 0$ such that for $t < T', i = 1, \dots, k$, and $x \in [i-1/4, i + 1/4]$ we have $\fh_t(x;  \de_i) = \fh^i_t(x;  \de_i)$. This follows from the more general coupling statement in Proposition \ref{P:landscape-mixing}. The equation \eqref{E:k-competing} then follows from the decay bounds in Proposition \ref{P:corollary107}, which guarantee that for all small enough $t$ we have 
\[
\max_{x \in \R} \fh_t(x;  h_k) = 	\max_{x \in \bigcup_{i=1}^k [i - 1/4, i + 1/4]} \fh_t(x;  \de_i), \qquad \max_{x \in \R} \fh^i_t(x;  \de_i) = \max_{x \in [i - 1/4, i + 1/4]} \fh^i_t(x;  \de_i). \qedhere \]
\end{proof}

\begin{corollary}
	\label{C:hausdorff-transfer}
In the setup above, define
$$
A_k = \{t \in (0, \infty) : \max \fh_t^1 = \dots = \max \fh_t^k \},
$$
and let $\ga_k$ be the supremum of the support of the random variable $\dim (A_k \cap [1/2, 1])$. Then
$
\al_k \ge \ga_k.
$
\end{corollary}

\begin{proof}
This is similar to the proof of Proposition \ref{P:alk}. By scale and shift invariance of $\sL$ we have that 
$$
\max_{x \in \R} q^{-1}\fh^i_{q^3t}(q^2x; \de_i) \eqd \max_{x \in \R} \fh_t(x; \de_i)
$$
for any $q > 0$ and $i \in \{1, \dots, k\}$, where the equality in distribution is as functions of $t$. Therefore $A_k \eqd q A_k$ for all $q > 0$. Since Hausdorff dimension is unchanged under linear dilations, this implies that the sequence $\dim (A_k \cap [2^q, 2^{q + 1}]), q \in \Z$ is stationary, and ergodic by Proposition \ref{P:tail-trivial}. Therefore for every $q \in \Z$, by the ergodic theorem we have that
$$
\dim (A_k \cap [0, 2^q]) = \sup_{n \in \Z \cap (-\infty, q-1]} \dim (A_k \cap [2^n, 2^{n + 1}]) = \ga_k
$$
almost surely. Now, in the coupling in Lemma \ref{L:h2-behaviour}, we have
$
A_k \cap [0, a] \sset \sT_{\ge k}(h_k) \cap [0, a],
$
for any $a< T$, so $\ga_k \le \al_k$. 
\end{proof}

Our remaining task is to estimate $\ga_k$. For this, we just need to understand the behaviour of the $k$ i.i.d.\ functions $F_i(t) = \max \fh_t^i$. Heuristically, each $F_i$ should be a H\"older-$1/3^-$ process with increments that are only weakly dependent, since we obtain $F_i(t + \ep)$ from $F_i(t)$ by metric composition of $\fh_t^i$ with an independent landscape increment. With this heuristic it is not difficult to estimate that $\ga_2 =  2/3$ and $\ga_3 = 1/3$. To make this heuristic precise the main technical tool we need is an estimate on the density of $F_i(t+ \ep)$ given $\fh_i(t)$. Equivalently, this amounts to estimating the density of $\max \fh_1$ from a fairly general initial condition. 

\begin{prop}
	\label{P:kpz-fp-prop}
	Let $h_0$ be any initial condition bounded above by $0$, and let $M = M(h_0) = \max_{x \in [-1, 1]} \fh_1(x; h_0)$. Then the random variable $M$ has CDF $F_{M}$ satisfying
	$$
	F_{M}(b) - F_M(a) \le c(b-a)\exp(e^{c (a^-)^{3/2}})
	$$
	for all $a < b$. Here $c > 0$ is a universal constant.
\end{prop}

The proof of Proposition \ref{P:kpz-fp-prop} follows from an estimate on the Fredholm determinant formula for $F_M$. As our proof is both a digression from our main themes and not novel or difficult, we leave it until Section \ref{A:fredholm-properties}.

We would like to use Proposition \ref{P:kpz-fp-prop} along with KPZ scale invariance to estimate conditional probabilities of the form
\begin{equation}
\label{E:hit-try}
\p\Big(|\max \fh^i_{t+s} - \max \fh^j_{t+s}| \le \ep \text{ for all } i \ne j \;\Big| \; |\max \fh^i_{t} - \max \fh^j_{t}| \le \ep \text{ for all } i \ne j \Big)
\end{equation}
over a wide range of $s, \ep$. However, there are two obstacles to doing this. The first is that Proposition \ref{P:kpz-fp-prop} only gives an estimate on the maximum over $[-1, 1]$ rather than over all of $\R$. The second, more serious, obstacle is that the bound on the Lebesgue density $F_M'(a)$ grows doubly exponentially as $a \to -\infty$\footnote{The double exponential bound here seems like a crude overestimate, but I had no luck finding a method to improve upon this. Note that it is possible (but rather technical) to show that $\sup_{h_0 : h_0 \le 0} F_M'(a)$ grows at least as a small power of $\log |a|$ as $a \to -\infty$. Similarly, for fixed $a$, $\sup_{h_0 : h_0 \le 0} F_{M_b}'(a)$ grows at least as a small power of $\log b$ as $b \to \infty$ if $M_b = \max \{\fh_1(x) : x \in [-b, b]\}$.}.

The workaround for dealing with these obstacles is to not 
to estimate \eqref{E:hit-try} directly, but rather first to thin out the set $A_k$ and then estimate a version of \eqref{E:hit-try} that corresponds to this thinning and is more amenable to a direct application of Proposition \ref{P:kpz-fp-prop}. The crucial miracle that allows this thinning to work is a law of the iterated logarithm for the KPZ fixed point that comes in to cancel out the double exponential in Proposition \ref{P:kpz-fp-prop}.
The next few lemmas set up this thinning.

\begin{lemma}
	\label{L:iterated-log}
	Let $\mathfrak{h}_t(x) = \sL(0,0; x, t)$ be the KPZ fixed point started from the narrow wedge initial condition, and let $M_t = \max \fh_t$. Define $f:[0, 1] \to \R$ by $f(t)= M_1 - M_{1 - t}$. Then almost surely,
	\begin{equation}
	\label{E:loglog}
	\inf_{t \in (0, 1/2]} \frac{f(t)}{(t\log \log (2 + t^{-1}))^{1/3}} > - \infty.
	\end{equation}
\end{lemma}

We prepare for the proof of Lemma \ref{L:iterated-log} with a simple bound on $\sL$.

\begin{lemma}
	\label{L:landscape-bdI}
	For an interval $I = [a, b], t \in (0, 1/2),$ and $s \in (1/2, 1)$, define
	$$
	Y_{I, s, t} := \max_{x \in I, r \in [0, t]} |\sL(0,0; x, s + r) - \sL(0,0; x, s)|.
	$$
	Then $\P(Y_{[a,b], s, t} > m t^{1/3} \log^{2/3}(t^{-1})) \le c(b-a + 1) e^{-d m^{3/2}}$ for universal $c, d > 0$.
\end{lemma}

\begin{proof}
The distribution of $Y_{I, 1, t}$ does not change if we shift $I$ and 
$$
Y_{[a, b], s, t} \le \max \{Y_{[i, i + 1], s, t} : i \in \{\fl{a}, \dots, \fl{b}\}\},
$$
so by a union bound it suffices to prove the lemma when $I = [0, 1]$. We can then apply Proposition \ref{P:modulus-bound} to the compact set $K = \{(0,0; x, 1/2 + r) \in \Rd : x, r \in [0, 1]\}$ to conclude the result.
\end{proof}

\begin{proof}[Proof of Lemma \ref{L:iterated-log}]
For all $n \in \N$, define $f^n(t) = M^n_1 - M^n_{1 - t}$ where $M^n_t = \max_{x \in [-n, n]} \fh_t(x)$
	First, by Lemma \ref{L:as-decay-wehave} all the $M_t, t \in [1/2, 1]$ are contained in a common compact set $[-N, N]$. Therefore it suffices to prove that for all $n \in \N$, \eqref{E:loglog} holds with $f^n$ in place of $f$.

	Let $S_t = \argmax_{x \in [-n, n]} \fh_t(x)$. We have $f^n(t) \ge \sL(S_{1-t}, 1-t; S_{1-t}, 1)$. Moreover, $S_{1-t}$ is independent of $\sL(\cdot, 1-t; \cdot, \cdot)$, and so
	\begin{equation}
	\label{E:f^n-bound}
\P(f^n(t) \le - m t^{1/3}) \le c e^{-d m^3}
	\end{equation}
	for absolute constants $c, d > 0$ since $\sL(S_{1-t}, 1-t; S_{1-t}, 1) \eqd t^{1/3} T$, where $T$ is a Tracy-Widom GUE random variable.
	 Next, with notation as in Lemma \ref{L:landscape-bdI}, for $0 < t < s \le 1/2$ we have
	\begin{align*}
	\min_{r \in [0, t]} f^n(s - r) = f^n(s) - \max_{r \in [0, t]} (M^n_{1 - s + r} - M^n_{1-s}) \ge f(s) - Y_{[-n, n], 1-s, t}.
	\end{align*}
	Therefore by \eqref{E:f^n-bound} and Lemma \ref{L:landscape-bdI}, for an absolute constant $c > 0$ we have
	$$
	\P\Big(\min_{r \in [0, 2^{-k}/k^2]} f(2^{-k}(1 + i/k^3) - r) < - c2^{-k/3}\log^{1/3}(k)\Big) < c k^{-5}
	$$
	for all $k \in \N, i \in \{1,\dots, k^3\}$. Applying the Borel-Cantelli lemma shows that the above event holds for only finitely many pairs $(k, i)$,  yielding the result.
	\end{proof}

\begin{lemma}
	\label{L:noniterated-log}
	Let $\mathfrak{h}_t$ be the KPZ fixed point started from an always decaying initial condition, and let $S_t =  \sup (\argmax \mathfrak{h}_t)$ and $I_t = \inf (\argmax \mathfrak{h}_t)$.
	Define $g:[0, 1] \to \R$ by $g(t) = |S_1 - S_{1 - t}| \vee |S_1 - I_{1-t}|$. Then almost surely, for every $\ep > 0$ we have
	$$
	\sup_{t \in (0, 1/2]} \frac{g(t)}{t^{2/3} \log^{16 + \ep}(t^{-1})} < \infty.
	$$
\end{lemma}

Note that by Lemma \ref{L:unique-max}, for any fixed $t$, almost surely $S_t = I_t$, so we could alternately the define $g(t)$ with $I_1$ instead of $S_1$.
\begin{proof}
First, by Lemma \ref{L:Bessel-lemma}, it suffices to prove the result when $\mathfrak{h}_t$ is started from the Bessel initial condition $-R$. Next, for the Bessel initial condition, by Theorem \ref{T:Bessel-IC} we have
$$
((S_{t + s} - S_t, I_{t + s} - S_t) : s \in (0, \infty)) \eqd ((S_s, I_s) : s \in (0, \infty))
$$
 for any $t \in [0, \infty)$. Moreover, we have the scale invariance
\begin{equation}
\label{E:Ss-bd}
((S_s, I_s) : s \in (0, \infty)) \eqd ( q^{-2} (S_{q^3 s}, I_{q^3 s}) : s \in (0, \infty))
\end{equation}
for any $q > 0$. Combining these facts gives that for any $k \in \N$ we have
\begin{align*}
\sup_{t \in [2^{-k}, 2^{-k + 1}]}& g(t) \\
&= \sup_{t \in [2^{-k}, 2^{-k + 1}]} |(S_1 - S_{1 - 2^{-k + 1}}) - (S_{1-t} - S_{1-2^{-k + 1}})| \vee |(S_1 - S_{1-2^{-k + 1}}) - (I_{1-t} - S_{1-2^{-k + 1}})| \\
&\eqd \sup_{t \in [2^{-k}, 2^{-k + 1}]} |S_{2^{-k + 1}} -S_{{2^{-k + 1}}-t}| \vee |S_{2^{-k + 1}}- I_{{2^{-k + 1}}-t}| \\
&\eqd 2^{-2k/3} \sup_{t \in [1, 2]} |S_{2} -S_{2-t}| \vee |S_2 -I_{2-t}| \\
&\le 2^{-2k/3} \Big(|S_2| + \sup_{t \in [0, 1]} |S_t| \vee |I_t|\Big).
\end{align*}
Therefore by a union bound over intervals of the form $[2^{-k}, 2^{-k + 1}]$, the result will follow if we can show that
\begin{equation}
\p(|S_2| + \sup_{t \in [0, 1]} |S_t| \vee |I_t| > m) \le cm^{-1/16}
\end{equation}
for a universal $c > 0$. By the invariance in \eqref{E:Ss-bd}, it suffices to prove a bound of the same form on $\sup_{t \in [0, 1]} |S_t| \vee |I_t|$. By Proposition \ref{P:corollary107} we have that
$$
|\sL(x, 0; y, s) + \frac{(x - y)^2}{s}| \le C \log^2(2 + |x| + |y|) 
$$
for all $s \in [0, 1]$ and a random $C$ satisfying $\p(C > m) \le ce^{-dm^{3/2}}$. Also, by a standard bound on Bessel processes (Lemma \ref{L:X0-explicit-bound}), the random variable
$$
X = \sup_{x \in \R} -R(x) + x^{1/4}
$$
satisfies $\p(X \ge m) \le cm^{-1/4}$. A computation then shows that for large enough $a$,
$$
\sup_{t \in [0, 1]} |S_t| \vee |I_t| \ge a \quad \implies \quad C > a^{1/8} \text{ or } X > a^{1/4}/2,
$$
from which the result follows.
\end{proof}

We use Lemmas \ref{L:iterated-log} and \ref{L:noniterated-log} to thin out the set $A_k$. Consider the KPZ fixed point $\fh$ started from a narrow wedge initial condition $\de_x$. Let $M_s, S_s, I_s$ be defined as in Lemmas \ref{L:iterated-log} and \ref{L:noniterated-log}, and for $t \in [1/2, 1]$ let
\begin{equation*}
L_1(\fh, t) := \inf_{s \in (0, t/2]} \frac{M_t - M_{t-s}}{(s\log \log (2 + s^{-1}))^{1/3}}, \qquad L_2(\fh,  t) = \sup_{s \in (0, t/2]} \frac{|S_t - S_{t-s}| \vee |S_t - I_{t-s}|}{s^{2/3} \log^{17}(s^{-1})}.
\end{equation*}
 For $\al > 0$, define
$$
B_\al(\fh) := \{t \in [1/2, 1] : L_1(\fh, t) > - \al, L_2(\fh, t) < \al\}.
$$
Rather than analyzing the set $A_k$ directly, we will analyze the sets $A_k \cap B_{\al, k}$ where 
$$
B_{\al, k} :=B_\al(\fh^1) \cap \dots \cap B_\al(\fh^k).
$$
Theorem \ref{T:lower-bd} will then follow immediately from combining Corollary \ref{C:hausdorff-transfer} with the following proposition.

\begin{prop}
	\label{P:HT-prop}
	For $k = 2, 3$, all large enough $\al > 0$ and any $\ga < (4-k)/3$, we have
$$
\p\lf(\dim(A_k \cap \close{B_{\al, k}}) \ge \ga \rg) > 0.
$$
Here $\close{B_{\al, k}}$ denotes the closure of $B_{\al, k}$.
\end{prop}

We build up to Proposition \ref{P:HT-prop} with a few lemmas.
First define sets 
$$
A_{k, \ep} = \{ t \in (0, \infty) : |\max \fh^i_t - \max \fh^j_t| \le \ep \text{ for all } i \ne j \in \{1, \dots, k\} \}.
$$
\begin{lemma}
\label{L:Akep-lemma}
Fix $k \in \N$. There exist $k$-dependent constants $c_1, c_2 > 0$ such that for all large enough $\al > 0$, and all $\ep \in (0, 1)$ and $t \in [1/2, 1]$ we have
$$
c_2 \ep^{k-1} \le \p(t \in A_{k, \ep} \cap B_{\al, k}) \le c_1 \ep^{k-1}.
$$
\end{lemma}

We will use the following simple lemma to prove Lemma \ref{L:Akep-lemma}.
\begin{lemma}
	\label{L:tight-diagonal}
	Let $\mu$ be a probability measure on $\R$ and let $b \ge 1$ be such that $\mu[-b, b] \ge 1/2$. Let $k \in \N, k \ge 2$, and let $\mu^k = \mu \otimes \cdots \otimes \mu$ denote the $k$-fold product measure. Then for all $\ep \in (0, 1)$ we have
	$$
	\mu^k(x \in \R^k : |x_i - x_j| \le \ep \text{ for all } i \ne j) \ge \frac{\ep^{k-1}}{2^{3k - 2} b^{k-1}}.
	$$
\end{lemma}

\begin{proof}
	First, we have
	\begin{align*}
	\mu^k(x \in \R^k : |x_i - x_j| \le \ep \text{ for all } i \ne j) &\ge \sum_{i\in \Z \cap [-\fl{b/\ep} - 1, \fl{b/ \ep}]} \lf(\mu[i\ep, (i+1)\ep) \rg)^k.
	\end{align*}
	We can recognize the right hand side above as the $k$th power of an $\ell^k$-norm on $\R^{2 \fl{b/\ep} + 2}$. H\"older's inequality then gives
	\begin{align*}
	\sum_{i\in \Z \cap [-\fl{b/\ep} - 1, \fl{b/\ep}]} \lf(\mu[i\ep, (i+1)\ep) \rg)^k \ge \frac{\lf(\sum_{i\in \Z \cap [-\fl{b/\ep} - 1, \fl{b/\ep}]} \mu[i\ep, (i+1)\ep) \rg)^k}{(2 \fl{b/\ep} + 2)^{k-1}}
	\end{align*}
	which gives the result after simplification, using that $\mu[-b, b] \ge 1/2$ and $2 \fl{b/\ep} + 2 \le 4b/\ep$.
\end{proof}

\begin{proof}[Proof of Lemma \ref{L:Akep-lemma}]
We start with the simpler upper bound. For this, it is enough to show that $\p (t \in A_{k, \ep}) \le c_1 \ep^{k-1}$ for all $\ep > 0$. Now, for all $t$, by KPZ scale invariance the random variables $\max \fh^1_t, \dots, \max \fh^k_t$ are independent and equal in distribution to $t^{1/3} T$, where $T$ is a Tracy-Widom GUE random variable. Therefore 
\begin{align*}
\p (t \in A_{k, \ep}) &\le \p(|\max \fh^1_t - \max \fh^j_t| \le \ep \text{ for all } j = 2, \dots, k-1) \\
&\le \max_{x \in \R} \p(T \in [x-\ep t^{1/3}, x + \ep t^{1/3}])^{k-1}\\
&\le c_1\ep^{k-1},
\end{align*}
where the final inequality uses that $T$ has a Lebesgue density that is bounded above.

We move on to the lower bound. Fix $i \in \{1, \dots, k\}$. For any fixed $t \in [1/2, 1]$ the processes $L_1(\fh^i, t), L_2(\fh^i, t) \in \R$ are finite almost surely by Lemmas \ref{L:iterated-log} and \ref{L:noniterated-log}. This extends to tightness of the families $L_1(\fh^i, t), L_2(\fh^i, t), t \in [1/2, 1]$ by KPZ scale invariance, so we can find $\al_0 > 0$ such that
\begin{equation}
\label{E:tBal}
\P(t \in B_\al(\fh^i)) \ge 1/2
\end{equation}
for all $t \in [1/2, 1], \al \ge \al_0$. Next, let $\mu_{t, \al}$ be the conditional distribution of $\max \fh^i_t$ on the event $\{t \in B_\al(\fh_i)\}$. Since $\max \fh^i_t \eqd t^{1/3} T$ for all $i$, by \eqref{E:tBal} there exists a $b > 0$ such that for all $t \in [1/2, 1], \al \ge \al_0$, we have \begin{equation}
\label{E:mutal}
\mu_{t, \al}[-b, b] \ge 1/2.
\end{equation}
Finally, we can write
\begin{equation*}
\p(t \in A_{k, \ep} \cap B_{\al, k}) = \mu^k_{t, \al}(x \in \R^n : |x_i - x_j| \le \ep \text{ for all } i \ne j) \prod_{i=1}^k \p B_\al(\fh^i),
\end{equation*}
which is bounded below by $c_2 \ep^{k-1}$ for all $t \in [1/2, 1], \al \ge \al_0$ by \eqref{E:tBal}, \eqref{E:mutal}, and Lemma \ref{L:tight-diagonal}.
\end{proof}

Lemma \ref{L:Akep-lemma} is the first moment bound that we will need to estimate the Hausdorff dimension of $A_k \cap \close{B_{\al, k}}$. We also need a complementary second moment bound. For this bound, it is essential that we work with $A_{k, \ep} \cap B_{\al, k}$ rather than just $A_{k, \ep}$.

\begin{lemma}
\label{L:second-moment}
Fix $k \in \N$ and $\al > 0$. There exists a constant $c = c(\al, k) > 0$ such that for all $\ep \in (0, 1), t, t + s \in [1/2, 1]$ we have
\begin{equation}
\label{E:ptsd}
\p(t + s\in A_{k, \ep} \cap B_{\al, k} \mid t \in A_{k, \ep} \cap B_{\al, k}) \le  \lf(\ep s^{-1/3} \rg)^{k-1} \exp(c \log^{1/2}(2 + s^{-1})).
\end{equation}
\end{lemma}

\begin{proof} 
Fix $t \in [1/2, 1]$ and for every $i = 1, \dots, k$ define the recentered KPZ fixed points 
$$
\tilde \fh^i_s(x) := \fh^i_{t+s}(S^i_t + x) - M^i_t, \qquad s > 0, x \in \R
$$ 
where $M^i_t, S^i_t$
are as in Lemmas \ref{L:iterated-log} and \ref{L:noniterated-log} for the $\fh^i$. The processes $\tilde \fh^1, \dots, \tilde \fh^k$ form a collection of independent KPZ fixed points run from independent non-positive initial conditions $T_{S^1_t} \fh^1_{t}, \dots, T_{S^k_t} \fh^k_{t}$. Now, conditionally on the event $\{t \in A_{k, \ep} \cap B_{\al, k}\}$ the event $\{t + s\in A_{k, \ep} \cap B_{\al, k}\}$ is contained in the event where:
\begin{enumerate}[label=(\roman*)]
	\item For all $i \in \{2, \dots, k\}$ we have 
	$$
	|\max \tilde \fh^1_s - \max \tilde \fh^j_s| \le 2\ep.
	$$ 
	\item Let $\tilde M^i_s, \tilde S^i_s$ be as in Lemmas \ref{L:iterated-log}, \ref{L:noniterated-log} for the $\tilde \fh^i$. Then for all $i \in \{1, \dots, k\}$, we have
	$$
	\tilde M^i_s > -\al(s \log \log (2 + s^{-1}))^{1/3} \qquad \mathand \qquad |\tilde S^i_s| < \al s^{2/3} \log^{17}(s^{-1}).
	$$
\end{enumerate}

Next, the event $\{t \in A_{k, \ep} \cap B_{\al, k}\}$ is measurable with respect to the $\sig$-algebra generated by $\fh^i_s, 0 \le s \le t$. Therefore by the Markov property for $\fh^i$, the conditional probability in \eqref{E:ptsd} is bounded above by the maximal possible probability of (i) and (ii) both occurring, where the maximum is taken over all independent KPZ fixed points $\tilde \fh^i(\cdot; h_i)$ started from non-positive deterministic initial conditions $h_1, \dots, h_k$.

For such a collection of KPZ fixed points $\tilde \fh^i(\cdot; h_i)$ and $j \in \Z$, let 
$$
\tilde M_s^i(j) = \max_{x \in [s^{2/3}j, s^{2/3}(j + 2)]} \tilde \fh_s^i(x; h_i).
$$ Then by a union bound,
\begin{align*}
\p( \text{(i) and (ii) above}) &\le \sum_{\substack{j_1, \dots, j_k \in \\
		\Z \cap [\fl{-\al \log^{17}(s^{-1})}, \fl{\al \log^{17}(s^{-1})}]}} \p A(j_1, \dots, j_k)
\end{align*} 
where $	A(j_1, \dots, j_k)$ is the set where
\begin{itemize}[nosep]
	\item $|\tilde M_s^i(j_i) - \tilde M_s^1(j_1)| \le 2\ep$ for all $i = 2, \dots, k$, and
	\item $\tilde M_s^i(j_i) \ge - \al (s\log \log (2 + s^{-1}) )^{1/3}$ for all $i =1 , \dots, k$.
\end{itemize}
Using the independence of the $\tilde \fh^i(\cdot; h_i)$, KPZ scale invariance, and spatial stationarity of $\sL$, we can bound $\p A(j_1, \dots, j_k)$ above by
$$
\sup_{h \le 0} \sup_{m \ge -\al (\log \log(2 + s^{-1}))^{1/3}}  \p\lf(\max_{x \in [-1, 1]} \fh_1(x; h) \in [m, m + 2 \ep s^{-1/3}]\rg)^{k-1}.
$$
By Proposition \ref{P:kpz-fp-prop}, this is bounded above by
$$
\lf(2 \ep s^{-1/3} c\exp(\exp (c \al^{3/2} (\log \log (2 + s^{-1}))^{1/2}\rg)^{k-1}.
$$
A bit of simplification then gives the result.
\end{proof}

To put everything together and prove Proposition \ref{P:HT-prop}, we use a consequence of Frostman's lemma for Hausdorff dimension and a compactness argument, \cite[Lemma 6.2]{schramm2011quantitative}.

\begin{lemma}
	\label{L:Dn-down-to-D}
	Let $D_1 \supset D_2 \supset D_3 \dots$ be a decreasing sequence of compact subsets of $[1/2, 1]$ with $D_n \cvgdown D$. Let  $\mu_n$ be a sequence of finite measures with $\supp(\mu_n) = D_n$ and suppose that there exists $c > 0$ such that 
	\begin{equation}
	\label{E:mom-2}
	\mu_n[1/2, 1] \ge 1/c, \qquad I_\ga(\mu_n) := \int_{1/2}^1 \int_{1/2}^1 |t - s|^{-\ga} d \mu_n(t) d \mu_n(s) \le c
	\end{equation}
	for all $n$.
	Then $\dim(D) \ge \ga$.
\end{lemma}

\begin{proof}[Proof of Proposition \ref{P:HT-prop}]
Fix $k \in \{2, 3\}$, and $\al > 0$ large enough so that Lemma \ref{L:Akep-lemma} holds.
By continuity of the KPZ fixed point for $t > 0$ we have that $A_{k, \ep} \cvgdown A_k$ almost surely as $\ep \to 0$. Therefore $\close{D_\ep} \cvgdown D$ almost surely as $\ep \to 0$, where $D_\ep := A_{k, \ep} \cap B_{\al, k}$ and $D := A_k \cap \close{ B_{\al, k}}$.
For $\ep \in (0, 1)$ define the random measure $\mu_\ep$ on $[1/2, 1]$ with support $\close{D_\ep}$ by letting
 \begin{equation*}
 \mu_\ep(A) = \frac{1}{\ep^{k-1}} \int_A \indic(t \in D_\ep) dt.
 \end{equation*} 
 By the lower bound in Lemma \ref{L:Akep-lemma}, there exists $c > 0$ such that $\E \mu_\ep [1/2, 1] \ge c$ for all $\ep > 0$. Next, by Lemma \ref{L:second-moment} and the upper bound in Lemma \ref{L:Akep-lemma} we have
 \begin{equation}
 \label{E:muepga}
 \begin{split}
\E I_{\ga}(\mu_\ep)
&= \frac{2}{\ep^{2k - 2}} \int_{1/2}^1 \int_0^{1- t} s^{-\ga} \p( t + s \in D_\ep \; | \; t \in D_\ep) \p(t \in D_\ep) ds dt \\
&\le  2c_1 \int_{0}^{1/2} s^{-\ga} s^{-(k-1)/3} \exp(c \log^{1/2}(2 + s^{-1})) ds.
\end{split}
 \end{equation} 
The final integral is independent of $\ep$ and finite whenever $\ga < (4-k)/3$. The particular case when $\ga = 0$ shows that $\E (\mu_\ep[1/2, 1])^2$ is uniformly bounded over $\ep \in (0, 1)$. In particular, by the lower bound on $\E \mu_\ep[1/2, 1]$ and the Paley-Zygmund inequality, for all large enough $c > 0$ we have
$$
\p(\mu_\ep[1/2, 1] > 1/c) \ge 1/c
$$
for all $\ep \in (0, 1)$. Therefore by \eqref{E:muepga}, for every $\ga < (4-k)/3$ there exists $c_\ga > 0$ such that
\begin{equation}
\label{E:quahep}
\p(\mu_\ep [1/2, 1] \ge 1/c_\ga, I_{\ga}(\mu_\ep) \le c_\ga) \ge 1/c_\ga. 
\end{equation}
Since $\p(A_n \text{ i.o.}) \ge \limsup_{n \to \infty} \p A_n$ for any sequence of events $A_n$, with probability $1/c_\ga$ we can find a (random) sequence $\ep_n \to 0$ such that the event in \eqref{E:quahep} holds for all $\ep_n, n \in \N$. Applying Lemma \ref{L:Dn-down-to-D} implies that $\dim(D) \ge \ga$ with probability $1/c_\ga > 0$, as desired.
\end{proof}

\begin{proof}[Proof of Theorem \ref{T:main}]
This is an immediate consequence of Proposition \ref{P:alk} and Theorems \ref{T:upper-bd} and \ref{T:lower-bd}.
\end{proof}

\section{The proof of Proposition \ref{P:kpz-fp-prop}}
\label{A:fredholm-properties}

\subsection{A Fredholm determinant formula}

To prove Proposition \ref{P:kpz-fp-prop} we will use a Fredholm determinant formula for the KPZ fixed point from \cite{matetski2016kpz}. We introduce only the minimum background that the reader will need to understand our manipulations. The interested reader should refer to \cite{matetski2016kpz, quastel2017totally} and references therein for more background.

To set up the formula, let $L^2(\R)$ be the space of square integrable functions $f:\R \to \mathbb C$. If $K$ is an operator on $L^2(\R)$ acting through its kernel $(Kf)(x) = \int f(y) K(x, y) dy$, its \textbf{Fredholm determinant} is given by 
$$
\det (I + K) = 1 + \sum_{n=1}^\infty \frac{1}{n!} \int \det[K(x_i, x_j)]_{i, j = 1}^n d x_1 \dots d x_n.
$$
This Fredholm determinant is well-defined and finite if $K$ has finite trace norm $\|K\|_1 = \operatorname{tr}(\sqrt{K^* K})$, where $\sqrt{K^* K}$ is the unique positive square root of the operator $K^* K$. If $\|K\|_1 < \infty$, we say that $K$ is trace class. We note for later use the standard estimates
\begin{align}
\label{E:AB-A-B}
\|ABC \|_1 \le \|A\|_{\operatorname{op}} \|C\|_{\operatorname{op}} \|B\|_1 &\le \|A\|_1 \|B\|_1 \|C\|_1,   \\
\label{E:det-cty}
|\det (I + A) - \det (I + B)| &\le \|A-B\|_1 e^{1 + \|A\|_1 + \|B\|_1},
\end{align}
see the first few sections of \cite[Chapter 1]{simon2005trace} for the first inequality and \cite[Theorem 3.4]{simon2005trace} for the second one.

Now, define $\operatorname{UC}$ to be the set of upper semicontinuous functions $f:\R \to \R \cup \{-\infty\}$ that are not identically equal to $-\infty$ and for all $x$ satisfy $f(x) \le \al  + \ga |x|$ for some constants $\al, \ga \in \R$. Let $\operatorname{LC} = \{f: - f\in \operatorname{UC}\}$. For any $f \in \operatorname{UC}, g \in \operatorname{LC},$ and $t > 0$ there are trace class \textbf{hypo} and \textbf{epi} operators $\mathbf{K}^{\operatorname{hypo}(f)}_t$ and $\mathbf{K}^{\operatorname{epi}(g)}_{-t}$ on $L^2(\R)$
such that
\begin{equation}
\label{E:Fredholm}
\P(\fh_t(y ; f) \le g(y) \text{ for all } y \in \R) = \det \lf(I - \mathbf{K}^{\operatorname{hypo}(f)}_{t/2}\mathbf{K}^{\operatorname{epi}(g)}_{-t/2}\rg).
\end{equation}
Note that the product $\mathbf{K}^{\operatorname{hypo}(f)}_{t/2}\mathbf{K}^{\operatorname{epi}(g)}_{-t/2}$ is trace class by \eqref{E:AB-A-B} so the Fredholm determinant is well-defined. The hypo and epi operators are related through the formula
\begin{equation}
\label{E:bfK-epi}
\mathbf{K}^{\operatorname{epi}(-f)}_{-t} = \rho \mathbf{K}^{\operatorname{hypo}(f)}_t \rho,  
\end{equation}
where $\rho$ is the reflection operator on $L^2(\R)$ given by $\rho f(x) := f(-x)$ (see \cite[Proposition 4.4]{matetski2016kpz}, so to describe \eqref{E:Fredholm} explicitly we just need to describe $\mathbf{K}^{\operatorname{hypo}(f)}_{t}$.
We give a definition of the operator $\mathbf{K}^{\operatorname{hypo}(f)}_t$ that matches that of \cite[Section 2]{corwin2021exceptional}, which differs from the presentation in \cite{matetski2016kpz} by a conjugation that does not affect the determinant. 
For $f \in \operatorname{UC}$ and $t > 0$ the operator $\mathbf{K}^{\operatorname{hypo}(f)}_{t}$ is a trace class limit of a product of operators:
\begin{equation}
\label{E:kpzhypof}
\mathbf{K}^{\operatorname{hypo}(f)}_{t} = \lim_{\ell_1 \to -\infty, \ell_2 \to \infty} \Ga_t (\mathbf{S}_{t, \ell_1})^* \mathbf{P}_{\ell_1, \ell_2}^{\operatorname{Hit} f} \mathbf{S}_{t, -\ell_2} \Ga_t.
\end{equation}
It remains to define the operators on the right side above. The operator $\Ga_t$ is multiplication by the function
$$
\Ga_t(z) = \exp(\ka_t \operatorname{sgn}(z) |z|^{3/2}),
$$
where $\ka_t > 0$ is any sufficiently small constant. Next, for $t > 0$ and $x \in \R$ let $\mathbf{S}_{t, x}$ be the integral operator with kernel $\mathbf{S}_{t, x}(z, y) = \mathbf{S}_{t, x}(z - y)$, where
$$
\mathbf{S}_{t, x}(z) = t^{-1/3} \exp(\tfrac{2x^3}{3t^2} -\tfrac{z x}t) \operatorname{Ai}(-t^{1/3}z + t^{-4/3} x^2),
$$
and $\operatorname{Ai}(z)$ is the classical Airy function. Finally, for $\ell_1 < \ell_2$ and $f \in \operatorname{UC}$ we define 
$
\mathbf{P}_{\ell_1, \ell_2}^{\operatorname{Hit} (f)} = I - \mathbf{P}_{\ell_1, \ell_2}^{\operatorname{No \; hit} (f)}, 
$
where $\mathbf{P}_{\ell_1, \ell_2}^{\operatorname{No \; hit} (f)}$ is an integral operator with kernel
$$
\mathbf{P}_{\ell_1, \ell_2}^{\operatorname{No \; hit} (f)}(u_1, u_2) = \P_{B(\ell_1) = u_1, B(\ell_2) = u_2}(B(y) > f(y) \text{ for } y \in [\ell_1, \ell_2]) \frac{1}{\sqrt{4 \pi(\ell_2 - \ell_1)}} e^{-\tfrac{(u_2 - u_1)^2}{4(\ell_2 - \ell_1)}}.
$$
Here the probability above denotes the probability that a Brownian bridge (of variance $2$) on the interval $[\ell_1, \ell_2]$ starting at $u_1$ and finishing at $u_2$ stays above the function $f$. Noting the Gaussian factor, the whole expression above can be equivalently thought of as a transition density for Brownian motion from $(u_1, \ell_1)$ to $(u_2, \ell_2)$, killed if it goes below $f$. While the limit definition \eqref{E:kpzhypof} is somewhat abstract, if $f$ is $-\infty$ outside of a compact interval $[-m_1, m_2]$, the expression under the limit in \eqref{E:kpzhypof} is the same for any $\ell_1 \le - m_1, m_2 \le \ell_2$, see the discussion after (4.1) in \cite{matetski2016kpz}. 

Now, by \eqref{E:Fredholm} the CDF $F_M(a) = \det(1 - \mathbf{K}^{\operatorname{hypo}(h_0)}_{1/2}\mathbf{K}^{\operatorname{epi}(g_a)}_{-1/2})$, where
$$
g_a(x) = 
\begin{cases}
a, \qquad x \in [-1, 1] \\
\infty, \qquad \text{ else.}
\end{cases}
$$
This expression for $F_M$ and some trace class estimates from \cite{matetski2016kpz} allow us to reduce Proposition \ref{P:kpz-fp-prop} to the following lemma.
\begin{lemma}
	\label{L:crux-lemma}
	There exists an absolute constant $c' > 0$ such that for any $a < b \in \R$ we have
	\begin{equation}
	\label{E:second-CDF-bd}
	\|\mathbf{K}^{\operatorname{hypo}(-g_{b})}_{1/2} -\mathbf{K}^{\operatorname{hypo}(- g_{a})}_{1/2}\|_1 \le c' (b-a) e^{c'(a^-)^{3/2}}.
	\end{equation}
\end{lemma}

\begin{proof}[Proof of Proposition \ref{P:kpz-fp-prop} given Lemma \ref{L:crux-lemma}]
	In \cite[Appendix A.1]{matetski2016kpz}, the authors show that there exists $c > 0$ such that for any $h_0 \le 0$ we have
	\begin{equation}
	\label{E:hypo-bound}
	\|\mathbf{K}^{\operatorname{hypo}(h_0)}_{1/2}\|_1 \le c.
	\end{equation}
	Therefore by \eqref{E:AB-A-B} and \eqref{E:det-cty}, for any $M= M(h_0)$ with $h_0 \le 0$ we have the estimate
	\begin{align*}
	&|F_{M}(b) - F_{M}(a)| 
	%\\
	%&\le \|\mathbf{K}^{\operatorname{hypo}(h_0)}_{1/2}(\mathbf{K}^{\operatorname{epi}(g_b)}_{-1/2} -\mathbf{K}^{\operatorname{epi}(g_a)}_{-1/2})\|_1 \exp \lf(1 + \|\mathbf{K}^{\operatorname{hypo}(h_0)}_{1/2}\mathbf{K}^{\operatorname{epi}(g_b)}_{-1/2}\|_1 + \|\mathbf{K}^{\operatorname{hypo}(h_0)}_{1/2}\mathbf{K}^{\operatorname{epi}(g_a)}_{-1/2}\|_1 \rg) \\
	%&
	\le c \|\mathbf{K}^{\operatorname{epi}(g_b)}_{-1/2} -\mathbf{K}^{\operatorname{epi}(g_a)}_{-1/2}\|_1 \exp \lf(1 + c\|\mathbf{K}^{\operatorname{epi}(g_b)}_{-1/2}\|_1 + c\|\mathbf{K}^{\operatorname{epi}(g_a)}_{-1/2}\|_1 \rg).
	\end{align*}
	Next, by \eqref{E:AB-A-B}, the definition \eqref{E:bfK-epi} and the fact that $\|\rho\|_{\operatorname{op}} =1$, we have the same inequality with $\mathbf{K}^{\operatorname{epi}(g_b)}_{-1/2}, \mathbf{K}^{\operatorname{epi}(g_a)}_{-1/2}$ replaced by $\mathbf{K}^{\operatorname{hypo}(-g_b)}_{1/2}, \mathbf{K}^{\operatorname{hypo}(-g_a)}_{1/2}$. Proposition \ref{P:kpz-fp-prop} then follows from Lemma \ref{L:crux-lemma} and the bound \eqref{E:hypo-bound} which gives that $\|\mathbf{K}^{\operatorname{hypo}(-g_a)}_{1/2}\|_1 \le c$ for any $a \ge 0$.
\end{proof}

\begin{proof}[Proof of Lemma \ref{L:crux-lemma}]
	The method here is the same as in \cite[Appendix A.1]{matetski2016kpz}. There are only superficial differences in the required calculations.
	Let $\mathbf{O}_{a, b} := \mathbf{K}^{\operatorname{hypo}(-g_a)}_{1/2} -\mathbf{K}^{\operatorname{hypo}(-g_b)}_{1/2}$. We have
	$
	\mathbf{O}_{a, b} = \Ga_{1/2} (\mathbf{S}_{1/2, -1})^* \mathbf{P}_{a, b} \mathbf{S}_{1/2, -1} \Ga_{1/2}
	$
	where 
	$$
	\mathbf{P}_{a, b} = \mathbf{P}_{-1, 1}^{\operatorname{No \; hit} (-g_b)} - \mathbf{P}_{-1, 1}^{\operatorname{No \; hit} (-g_a)}.
	$$
	Since $g_a, g_b$ are constant on $[-1, 1]$, the kernel of $\mathbf{P}_{a, b}$ can be computed explicitly by the reflection principle:
	\begin{align*}
	\mathbf{P}_{a, b}(w, v) &= \mathbf{1}(w \wedge v > -a) \frac{1}{\sqrt{8\pi}} \lf(e^{-(w + v + 2a)^2/8} - e^{-(w + v + 2b)^2/8}\rg) \\
	&+\mathbf{1}(w \wedge v \in (-b, -a]) \frac{1}{\sqrt{8\pi}} \lf(e^{-(w - v)^2/8} - e^{-(w + v + 2b)^2/8}\rg).
	\end{align*} 
	Now, the key idea is to view $\mathbf{O}_{a,b}$ as an integral of rank one operators integrated against the positive measure $\mathbf{P}_{a, b}(w, v)dwdv$. Indeed, we can write out the kernel of $\mathbf{O}_{a,b}$ as:
	\begin{equation*}
	\mathbf{O}_{a, b}(t, u) =  \int \Ga_{1/2}(t) \mathbf{S}_{1/2, -1}(w, t) \Ga_{1/2}(u) \mathbf{S}_{1/2, -1}(v, u)  \mathbf{P}_{a, b}(w, v)  dwdv.
	\end{equation*}
	Let $\mathbf{R}_{w, v}$ be the operator with kernel
	$$
	\mathbf{R}_{w, v}(t, u) := R_w(t) R_v(u), \quad \text{ where } \quad R_w(t) = \Ga(t) \mathbf{S}_{1/2, -1}(w, t).
	$$
	We can then upper bound $\|\mathbf{O}_{a, b}\|_1$ by an integral of the trace norms of $\mathbf{R}_{w, v}$:
	\begin{equation}
	\label{E:Oab-bound}
	\|\mathbf{O}_{a, b}\|_1 \le \int \|\mathbf{R}_{w, v}\|_1 \mathbf{P}_{a, b}(w, v) dw dv = \int \|R_w\|_2 \|R_v\|_2 \mathbf{P}_{a, b}(w, v) dw dv.
	\end{equation}
	For the second equality we have used that $\mathbf{R}_{v, w}$ is rank one, and hence its trace norm is simply the product of the $L^2(\R)$-norms of the functions $R_v, R_w$. At this point it just remains to estimate the final integral. In the remainder of the proof the constant $c > 0$ depends only on $\kappa_{1/2}$ and may change from line to line.
	
	To bound $\|R_w\|_2$ we will use the standard Airy function bound
	$\operatorname{Ai}(z) \le c\exp(-\frac{2}{3}(z^+)^{3/2})$, see \cite{abramowitz1972handbook}, formulas 10.4.59–10.4.60. This yields the bound
	\begin{align}
	\nonumber
	\|R_w\|^2_2 &\le  c \int_\R \exp (2 \operatorname{sgn}(t)\ka_{1/2} |t|^{3/2} -4(t - w) - \frac{2\sqrt{2}}{3}((t-w)^+)^{3/2}) dt.
	\end{align}
	As long as $\ka_{1/2} < \sqrt{2}/3$,  this gives the bound
	$$
	\|R_w\|^2_2 \le c \exp(2w + c (w^+)^{3/2}).
	$$
	We can now turn to bounding the right side of \eqref{E:Oab-bound}. We first look at at the portion of the integral where $w \wedge v \in (-b, -a]$. Using the fact that the integrand is symmetric in $w$ and $v$ and making the change of variables $w \mapsto x = w - v$, we can write
	\begin{align*}
	\int_{w \wedge v \in (-b, -a)} \|R_w\|_2 \|R_v\|_2 \mathbf{P}_{a, b}(w, v) dw dv = 2\int_{-b}^{-a} \int_0^\infty \|R_{x + v}\|_2 \|R_v\|_2 \mathbf{P}_{a, b}(x + v, v) dx dv.
	\end{align*}
	Using the bound on $\|R_w\|_2$ and the fact that $\mathbf{P}_{a, b}(x + v, v) \le e^{-x^2/8}$ in the domain of integration, this is bounded above by
	\begin{align*}
	c \int_{-b}^{-a} \int_0^\infty  \exp(2x + 4v + c (v^+)^{3/2} + &c((x + v)^+)^{3/2} - x^2/8) dx dv \\
	&\le c (b-a) \exp(c  (a^-)^{3/2}).
	\end{align*}
	To complete the proof of the lemma, it just remains to prove the same bound on the portion of the integral where $w \wedge v > - a$. In this region we use that
	$$
	\mathbf{P}_{a, b}(w, v) \le c (b- a) e^{-(w + v + 2a)^2/9},
	$$
	which follows from a straightforward bound on the derivative of $e^{-(u_1 + u_2 + 2a)^2/8}$ with respect to $a$. Then using the bound on $\|R_w\|_2$ and the change of variables $x = w + a, y = v +a$ we have
	\begin{align*}
	&\int_{-a}^\infty \int_{-a}^\infty \|R_w\|_2 \|R_v\|_2 \mathbf{P}_{a, b}(w, v) dw dv \\
	&\le c (b - a)\int_0^\infty \int_0^\infty \exp\Big(2(x +y) - 4a + c ((x - a)^+)^{3/2} + c((y - a)^+)^{3/2} -(x + y)^2/9\Big) dx dy \\
	&\le c (b - a) \exp(c (a^-)^{3/2}),
	\end{align*}
	as desired.
\end{proof}

\bibliographystyle{alpha}
\bibliography{bibliography}

\begin{thebibliography}{CHHM21}

\bibitem[AS72]{abramowitz1972handbook}
Milton Abramowitz and Irene~A Stegun.
\newblock {\em Handbook of mathematical functions: with formulas, graphs, and
  mathematical tables}, volume~55.
\newblock Dover publications New York, 1972.

\bibitem[BBS21]{balazs2021local}
M{\'a}rton Bal{\'a}zs, Ofer Busani, and Timo Sepp{\"a}l{\"a}inen.
\newblock Local stationarity in exponential last-passage percolation.
\newblock {\em Probability Theory and Related Fields}, 180(1):113--162, 2021.

\bibitem[BDJ99]{baik1999distribution}
Jinho Baik, Percy Deift, and Kurt Johansson.
\newblock On the distribution of the length of the longest increasing
  subsequence of random permutations.
\newblock {\em Journal of the American Mathematical Society}, 12(4):1119--1178,
  1999.

\bibitem[BF22]{busani2020universality}
Ofer Busani and Patrik Ferrari.
\newblock Universality of the geodesic tree in last passage percolation.
\newblock {\em Annals of Probability}, To appear, 2022+.

\bibitem[BG16]{borodin2016lectures}
Alexei Borodin and Vadim Gorin.
\newblock Lectures on integrable probability.
\newblock {\em Probability and statistical physics in St. Petersburg},
  91:155--214, 2016.

\bibitem[BGH21]{basu2021fractal}
Riddhipratim Basu, Shirshendu Ganguly, and Alan Hammond.
\newblock Fractal geometry of {Airy$_2$} processes coupled via the {A}iry
  sheet.
\newblock {\em Annals of Probability}, 49(1):485--505, 2021.

\bibitem[BGH22]{bates2019hausdorff}
Erik Bates, Shirshendu Ganguly, and Alan Hammond.
\newblock Hausdorff dimensions for shared endpoints of disjoint geodesics in
  the directed landscape.
\newblock {\em Electron. J. Probab.}, 27:1--44, 2022.

\bibitem[BGZ21]{basu2021temporal}
Riddhipratim Basu, Shirshendu Ganguly, and Lingfu Zhang.
\newblock Temporal correlation in last passage percolation with flat initial
  condition via {B}rownian comparison.
\newblock {\em Communications in Mathematical Physics}, 383(3):1805--1888,
  2021.

\bibitem[CH14]{CH}
Ivan Corwin and Alan Hammond.
\newblock Brownian {G}ibbs property for {A}iry line ensembles.
\newblock {\em Inventiones mathematicae}, 195(2):441--508, 2014.

\bibitem[CHH19]{CHH20}
Jacob Calvert, Alan Hammond, and Milind Hegde.
\newblock Brownian structure in the {KPZ} fixed point.
\newblock arXiv:1912.00992, 2019.

\bibitem[CHHM21]{corwin2021exceptional}
Ivan Corwin, Alan Hammond, Milind Hegde, and Konstantin Matetski.
\newblock Exceptional times when the {KPZ} fixed point violates {J}ohansson's
  conjecture on maximizer uniqueness.
\newblock {\em arXiv preprint arXiv:2101.04205}, 2021.

\bibitem[Cor12]{corwin2012kardar}
Ivan Corwin.
\newblock The {K}ardar--{P}arisi--{Z}hang equation and universality class.
\newblock {\em Random matrices: Theory and applications}, 1(01):1130001, 2012.

\bibitem[DOV18]{DOV}
Duncan Dauvergne, Janosch Ortmann, and B{\'a}lint Vir{\'a}g.
\newblock The directed landscape.
\newblock arXiv:1812.00309, 2018.

\bibitem[DSV20]{dauvergne2020three}
Duncan Dauvergne, Sourav Sarkar, and B{\'a}lint Vir{\'a}g.
\newblock Three-halves variation of geodesics in the directed landscape.
\newblock {\em arXiv preprint arXiv:2010.12994}, 2020.

\bibitem[DV21a]{dauvergne2021bulk}
Duncan Dauvergne and B{\'a}lint Vir{\'a}g.
\newblock Bulk properties of the {A}iry line ensemble.
\newblock {\em Annals of Probability}, 49(4):1738--1777, 2021.

\bibitem[DV21b]{dauvergne2021scaling}
Duncan Dauvergne and B{\'a}lint Vir{\'a}g.
\newblock The scaling limit of the longest increasing subsequence.
\newblock {\em arXiv preprint arXiv:2104.08210}, 2021.

\bibitem[DZ21]{dauvergne2021disjoint}
Duncan Dauvergne and Lingfu Zhang.
\newblock Disjoint optimizers and the directed landscape.
\newblock {\em arXiv preprint arXiv:2102.00954}, 2021.

\bibitem[FQR13]{flores2013endpoint}
Gregorio~Moreno Flores, Jeremy Quastel, and Daniel Remenik.
\newblock Endpoint distribution of directed polymers in {$1+ 1$} dimensions.
\newblock {\em Communications in Mathematical Physics}, 317(2):363--380, 2013.

\bibitem[FS10]{ferrari2010random}
Patrik~L Ferrari and Herbert Spohn.
\newblock Random growth models.
\newblock arXiv:1003.0881, 2010.

\bibitem[GH20]{ganguly2020stability}
Shirshendu Ganguly and Alan Hammond.
\newblock Stability and chaos in dynamical last passage percolation.
\newblock {\em arXiv preprint arXiv:2010.05837}, 2020.

\bibitem[Ham19]{hammond2019patchwork}
Alan Hammond.
\newblock A patchwork quilt sewn from brownian fabric: Regularity of polymer
  weight profiles in brownian last passage percolation.
\newblock {\em Forum of Mathematics, Pi}, 7, 2019.

\bibitem[Ham20]{hammond2020exponents}
Alan Hammond.
\newblock Exponents governing the rarity of disjoint polymers in brownian last
  passage percolation.
\newblock {\em Proceedings of the London Mathematical Society},
  120(3):370--433, 2020.

\bibitem[Ham22]{hammond2016brownian}
Alan Hammond.
\newblock Brownian regularity for the {A}iry line ensemble, and multi-polymer
  watermelons in brownian last passage percolation.
\newblock {\em Mem. Amer. Math. Soc.}, To appear, 2022+.

\bibitem[Joh00]{johansson2000shape}
Kurt Johansson.
\newblock Shape fluctuations and random matrices.
\newblock {\em Communications in Mathematical Physics}, 209(2):437--476, 2000.

\bibitem[Joh03]{johansson2003discrete}
Kurt Johansson.
\newblock Discrete polynuclear growth and determinantal processes.
\newblock {\em Communications in Mathematical Physics}, 242(1-2):277--329,
  2003.

\bibitem[MQR21]{matetski2016kpz}
Konstantin Matetski, Jeremy Quastel, and Daniel Remenik.
\newblock The {KPZ} fixed point.
\newblock {\em Acta Mathematica}, 227(1):115--203, 2021.

\bibitem[NQR20]{nica2020one}
Mihai Nica, Jeremy Quastel, and Daniel Remenik.
\newblock One-sided reflected brownian motions and the {KPZ} fixed point.
\newblock {\em Forum of Mathematics, Sigma}, 8:e63, 2020.

\bibitem[Pim14]{pimentel2014location}
Leandro Pimentel.
\newblock On the location of the maximum of a continuous stochastic process.
\newblock {\em Journal of Applied Probability}, 51(1):152--161, 2014.

\bibitem[Pim18]{pimentel2018local}
Leandro Pimentel.
\newblock Local behaviour of {A}iry processes.
\newblock {\em Journal of Statistical Physics}, 173(6):1614--1638, 2018.

\bibitem[PS02]{prahofer2002scale}
Michael Pr{\"a}hofer and Herbert Spohn.
\newblock Scale invariance of the {PNG} droplet and the {A}iry process.
\newblock {\em Journal of statistical physics}, 108(5-6):1071--1106, 2002.

\bibitem[QM17]{quastel2017totally}
Jeremy Quastel and Konstantin Matetski.
\newblock From the totally asymmetric simple exclusion process to the {KPZ}
  fixed point.
\newblock {\em arXiv preprint arXiv:1710.02635}, 2017.

\bibitem[QS20]{quastel2020convergence}
Jeremy Quastel and Sourav Sarkar.
\newblock Convergence of exclusion processes and {KPZ} equation to the {KPZ}
  fixed point.
\newblock {\em arXiv preprint arXiv:2008.06584}, 2020.

\bibitem[Qua11]{quastel2011introduction}
Jeremy Quastel.
\newblock Introduction to {KPZ}.
\newblock In {\em Current {D}evelopments in {M}athematics}. International Press
  of Boston, 2011.

\bibitem[Rom15]{romik2015surprising}
Dan Romik.
\newblock {\em The surprising mathematics of longest increasing subsequences},
  volume~4.
\newblock Cambridge University Press, 2015.

\bibitem[RY13]{revuz2013continuous}
Daniel Revuz and Marc Yor.
\newblock {\em Continuous martingales and Brownian motion}, volume 293.
\newblock Springer Science \& Business Media, 2013.

\bibitem[Sim05]{simon2005trace}
Barry Simon.
\newblock {\em Trace ideals and their applications}.
\newblock Number 120. American Mathematical Soc., 2005.

\bibitem[SS10]{schramm2011quantitative}
Oded Schramm and Jeffrey Steif.
\newblock Quantitative noise sensitivity and exceptional times for percolation.
\newblock {\em Annals of Mathematics}, 171(2):619--672, 2010.

\bibitem[SV21]{sarkar2020brownian}
Sourav Sarkar and B{\'a}lint Vir{\'a}g.
\newblock Brownian absolute continuity of the {KPZ} fixed point with arbitrary
  initial condition.
\newblock {\em Annals of Probability}, 49(4):1718--1737, 2021.

\bibitem[Vir20]{virag2020heat}
B{\'a}lint Vir{\'a}g.
\newblock The heat and the landscape {I}.
\newblock {\em arXiv preprint arXiv:2008.07241}, 2020.

\end{thebibliography}

\appendix
\section{Proof of Proposition \ref{P:landscape-mixing}}
\label{A:landscape-properties}

For the proof of Proposition \ref{P:landscape-mixing} we will need to approximate the directed landscape with a prelimiting last passage model. We use Poisson last passage percolation to avoid complications from working on a lattice. Let $P_{\la}$ be an intensity-$\la$ Poisson process on $\R^2$. We say that a continuous function $f:[s, t] \to \R$ with $f(s) = x, f(t) = y$ is a path from $p = (x, s)$ to $q =(y, t)$.

For a path $f:[s, t] \to \R$ and parameters $\ell \ge 0, \la, \chi > 0$ define the length
$$
\|f\|_{\la, \ell, \chi} = \frac{\# (P_\la \cap \fg f)}{\chi}  - \frac{\ell (t-s)}{\chi},
$$  
where $\fg f = \{(f(r), r) : r \in [s, t]\}$ is the \textbf{graph} of $f$. Then for $(p, q) \in \Rd$ and $m > 0$ define
$$
d_{m, \la, \ell, \chi}(p, q) = \max \{ \|f\|_{\la, \ell, \chi} : f \text{ is a path from } p \text{ to } q, f \text{ is $m$-Lipschitz}\}.
$$
 We say that an $m$-Lipschitz path $f$ from $p$ to $q$ is a $d_{m, \la, \ell, \chi}$-geodesic if $d_{m, \la, \ell, \chi}(p, q) = \|f\|_{\la, \ell, \chi}$.

The case when $\la = m = \chi = 1$ and $\ell = 0$ gives us the usual definition of Poisson last passage percolation, rotated by $\pi/4$. The introduction of parameters allows us a simple way to introduce rescaling into the setup without shifting the underlying space $\Rd$. We then have the following result.

\begin{theorem}[First part of Corollary 13.12, \cite{dauvergne2021scaling}]
	\label{T:poisson}
For $n \in \N$, let $d^n := d_{n^{1/3}/2, 4 n^{5/3}, 2n, n^{1/3}}$. Then in some coupling we have
$
d^n \to \sL
$
almost surely, uniformly over compact sets in $\Rd$. 
\end{theorem}

In the context of Theorem \ref{T:poisson}, we can also understand what happens to geodesics in the limit. For a path $f:[s, t] \to \R$, define its $\sL$-length
\begin{equation*}
\|f\|_\sL = \inf_{k \in \N} \inf_{s = r_0 < r_1 < \dots < r_k = t} \sum_{i=1}^k \sL(f(r_{i-1}), r_{i-1}; f(r_i), r_i).
\end{equation*}
By the triangle inequality \eqref{E:triangle-ineq}, for any path $f$ from $p$ to $q$ we have $\|f\|_\sL \le \sL(p, q)$. We call $f$ an $\sL$-geodesic if $\|f\|_\sL = \sL(p, q)$. Note that the triangle inequality \eqref{E:triangle-ineq} becomes an equality for a triple of points on a common geodesic. 

\begin{theorem}[Second part of Corollary 13.12, \cite{dauvergne2021scaling}]
	\label{T:poisson-geods}
In the coupling in Theorem \ref{T:poisson-geods}, almost surely the following holds. If $f_n$ is any sequence of $d^n$-geodesics from $p_n$ to $q_n$ and $(p_n, q_n) \to (p, q)$ in $\Rd$ then the graphs $\fg f_n$ are precompact in the Hausdorff topology and any subsequential limit $\Ga$ of $\fg f_n$ is the graph of an $\sL$-geodesic from $p$ to $q$.
\end{theorem}

We can use these two convergence theorems to prove Proposition \ref{P:landscape-mixing}.

\begin{proof}
For every $n$ define $k + 1$ intensity $4n^{5/3}$-Poisson processes, $P^n_0, P^n_1, \dots, P^n_k$, coupled so that $P^n_1, \dots, P^n_k$ are independent, 
\begin{equation}
\label{E:Pn0}
P^n_0|_{[i-1/3, i + 1/3] \X [0, 1]} = P^n_i |_{[i-1/3, i + 1/3] \X [0, 1]}
\end{equation}
for all $i \ge 1$, and $P^n_0$ is independent of $P^n_1, \dots, P^n_k$ outside of $\bigcup_{i=1}^k [i-1/3, i + 1/3] \X [0, 1]$. Let $d^n_i$ be the versions of $d_{n^{1/3}/2, 4 n^{5/3}, 2n, n^{1/3}}$ defined using the different $P^n_i$. Also, for every $i = 1, \dots, k$ and $j \in \N$ for a directed landscape $\sL$ let
\begin{align*}
G_{i, j}(\sL) = \sup \{|g(x) - i|: g \text{ is an}& \text{ $\sL$-geodesic between } (x, s), (y, t) \\ \text{ with } &(x, s; y, t) \in K_{i, 1/j}\}. 
\end{align*}
Here $K_{i, 1/j}$ is as in the statement of Proposition \ref{P:landscape-mixing}.
We similarly define $G_{i, j}(d^n_{i'})$ for $i' = 0, \dots, k$. Now, the decay bound in Proposition \ref{P:corollary107} and the fact that the triangle inequality is an equality along geodesics implies that 
$G_{i, j}(\sL) \to 1/4$ almost surely as $j \to \infty$ for all $i$. Also, Theorem \ref{T:poisson-geods} implies that
\begin{equation}
\label{E:Gsup}
\limsup_{n \to \infty} G_{i, j}(d^n) \le G_{i, j}(\sL), \qquad \text{ and so } \qquad \limsup_{j \to \infty} \limsup_{n \to \infty} G_{i, j}(d^n) \le 1/4.
\end{equation}
almost surely for all $j$ in the coupling from that theorem.

Now, using Theorem \ref{T:poisson} and \eqref{E:Gsup} the random variables $G_{i, j}(d^n_{i'}), d^n_{i'}(u), i \in \{1, \dots, k\}, i' \in \{0, \dots, k\}, u \in \Q^4 \cap \Rd$ are jointly tight. Therefore we can find a subsequence and a coupling where all these random variables converge almost surely along that subsequence. 

By Theorem \ref{T:poisson}, in this coupling there are directed landscapes $\sL_0, \dots, \sL_k$ such that $d^n_{i}(u) \to \sL_i(u)$ for all $u \in \Q^4 \cap \Rd$ and all $i = 0, \dots, k$. The landscapes $\sL_1, \dots, \sL_k$ are independent. Also, by \eqref{E:Gsup} there exists a random $J$ such that for all $i, i'$ we have $G_{i, J}(d^n_{i'}) \le 1/3$ for all large enough $n$. Therefore by \eqref{E:Pn0} we have $d^n_0 = d^n_i$ on $K_{i, 1/J}$ for all $i \in \{1, \dots, k\}$, and so $\sL_0 = \sL_i$ on $K_{i, 1/J}$ for all $i \in \{1, \dots, k\}$, as desired. 
\end{proof}
	
\section{Facts about Bessel processes}
\label{A:bessel-properties}

\begin{lemma}
	\label{L:bessel-1}
	Let $R_a:[0,\infty) \to [0, \infty)$ be a Bessel-$3$ process with $R_a(0) = a$ and define the random variable
	$$
	X_a:=\inf_{x \ge 0} R_a(x) - x^{1/4} - \sqrt{a}
	$$
	Then the family of random variables $X_a \wedge 0, a \ge 0$ is tight.
\end{lemma}

\begin{proof}
	To prove the lemma, we will use that conditional on $R_a(t)$, the minimum of $R_a$ on $[t, \infty)$ is uniform on $[0, R_a(t)]$, see \cite[Chapter VI, Corollary 3.4]{revuz2013continuous}. Combining this with the representation of $R_a$ as the magnitude of a $3$-dimensional Brownian motion (of variance $2$) started at the point $(a, 0, 0)$ gives that for fixed $t \ge 0$,
	\begin{equation}
	\label{E:normal-rep}
\min_{x \ge t} R_a(x) \eqd U \sqrt{(N_1 + a)^2 + N_2^2 + N_3^2}
	\end{equation}
	where $U \sim U(0, 1), N_i \sim N(0, 2t)$ and all random variables are independent.
	 
	We first use this to show $\tilde X_0 :=\inf_{x \ge 0} R_0(x) - 2x^{1/4} \ne -\infty$ almost surely. Consider the sequence of random variables
	$$
	I_n = \min_{x \in [2^n, \infty)} 2^{-n/2} R_0(x), \qquad n = 0, 1, 2, \dots
	$$
	By \eqref{E:normal-rep}, this is a sequence of identically distributed non-negative random variables with Lebesgue density bounded above. Therefore by a union bound,
	$
	I :=\inf_{n \ge 1} n^2 I_n
	$
	is non-zero almost surely, and hence so is the random variable $\inf_{x \ge 1} \log^2(x) x^{-1/2} R_0(x)$. Noting also that $\inf_{x \in [0, 1]} R_0(x) - 2x^{1/4} \ge -2$ gives that $\tilde X_0 \ne -\infty$. 
	
	Now, for $a \ge 0$ define
	$$
	X^1_a := \inf_{x \le a^2} R_a(x) - 2\sqrt{a}, \qquad X_a^2:= \inf_{x \ge a^2} R_a(x) - 2x^{1/4}
	$$
	so that $X_a \ge X^1_a \wedge X^2_a$. By \eqref{E:normal-rep}, we have $\min_{x \ge 0} R_a(x) \eqd a U$, so
	$
	aU - 2\sqrt{a} \preceq X^1_a$
	where $\preceq$ is stochastic ordering. The random variables $(aU - 2\sqrt{a}) \wedge 0, a \ge 0$ are tight, and hence so are the random variables $X^1_a \wedge 0, a \ge 0$. Also, since the Bessel processes $R_a$ are stochastically increasing in $a$, the random variables $X^2_a \wedge 0, a \ge 0$ are tight since $X^2_0 = \tilde X_0$ is not $-\infty$ almost surely.
	\end{proof}

\begin{lemma}
	\label{L:X0-explicit-bound}
	Let $R:[0, \infty) \to [0, \infty)$ be a Bessel-$3$ process started at $0$, and let $X_0 = \inf_{x \ge 0} R(x) - x^{1/4}$. Then for some $c > 0$ we have
	$$
	\P(X_0 < -m) \le cm^{-1/4}.
	$$
\end{lemma}

\begin{proof}
	Using the notation from the previous proof, by a union bound we have
	\begin{align*}
\P(X_0 < -m) \le \sum_{n= \fl{\log_2 m}}^\infty \P(\min_{y \in [2^n, 2^{n+1}]} R_0(x) - 2^{(n+1)/4} < 0) \le \sum_{n= \fl{\log_2 m}}^\infty \P(2^{n/2} I_n < 2^{(n+1)/4}).
	\end{align*}
	This is bounded above by $c m^{-1/4}$ since the $I_n$ are identically distributed with Lebesgue density bounded above by \eqref{E:normal-rep}.
\end{proof}

\end{document}